\documentclass[11pt]{amsart} 
\usepackage[utf8]{inputenc}
\usepackage[a4paper]{geometry}
\usepackage{color}
\usepackage{amssymb}
\usepackage{mathrsfs}
\usepackage[all]{xy}
\usepackage{hyperref}
\usepackage{comment}
\usepackage[skip=2pt plus1pt, indent=20pt]{parskip}
\newtheorem{thm}{Theorem}[section]
\newtheorem{cor}[thm]{Corollary}
\newtheorem{lem}[thm]{Lemma}
\newtheorem{prop}[thm]{Proposition}
\theoremstyle{definition}

\newtheorem{rem}[thm]{Remark}

\newtheorem{example}[thm]{Example}
\newtheorem{question}{Question}
\numberwithin{equation}{section}
\newcommand{\fbl}{\text{FBL}}
\newcommand{\fvl}{\text{FVL}}
\newcommand{\vertiii}[1]{{\left\vert\kern-0.25ex\left\vert\kern-0.25ex\left\vert #1 
		\right\vert\kern-0.25ex\right\vert\kern-0.25ex\right\vert}}

\newcommand\PS{{\normalfont\textsf{PS}}_\textsf{2}}
		
\title[Complemented subspaces of Banach lattices]{Complemented subspaces of Banach lattices}

\author[de Hevia]{David de Hevia}
\address{Instituto de Ciencias Matem\'aticas (CSIC-UAM-UC3M-UCM)\\
Consejo Superior de Investigaciones Cient\'ificas\\
C/ Nicol\'as Cabrera, 13--15, Campus de Cantoblanco UAM\\
28049 Madrid, Spain.}
\email{david.dehevia@icmat.es}

\author[Tradacete]{Pedro Tradacete}
\address{Instituto de Ciencias Matem\'aticas (CSIC-UAM-UC3M-UCM)\\
Consejo Superior de Investigaciones Cient\'ificas\\
C/ Nicol\'as Cabrera, 13--15, Campus de Cantoblanco UAM\\
28049 Madrid, Spain.}
\email{pedro.tradacete@icmat.es}



\begin{document}

\begin{abstract}
    We survey recent developments on the structure of complemented subspaces of Banach lattices, including in particular the construction of a complemented subspace of a $C(K)$-space which is not linearly isomorphic to any Banach lattice. Motivated by this, several natural questions and directions of future research are presented. We provide an approach to some of these problems using tools from the theory of free Banach lattices. 
\end{abstract}

\maketitle
\nocite{*}
\section{Introduction}

The spaces $L_p(\Omega,\Sigma,\mu)$, for $1\leq p\leq \infty$, and $C(K)$, for a compact Hausdorff space $K$, are undoubtedly among the most important and interesting classes of classical Banach spaces. The study of their subspaces has been the object of intensive research over several decades which has significantly increased our understanding of the geometry of Banach spaces. However, some basic questions about them still remain a mystery. This is the case for instance of characterizing their complemented subspaces. In other words, determining which decompositions of the form $L_p(\Omega,\Sigma,\mu)=X\oplus Y$ (respectively, $C(K)=X\oplus Y$) can be implemented. Since these classical spaces are particular instances of Banach lattices, our aim in this note is to present recent progress and new approaches in the study of complemented subspaces of Banach lattices, survey the current state of certain open fundamental questions and raise some more to stimulate further research.

A central open question in the matter is the \textit{Complemented Subspace Problem for Banach lattices} (CSP, for short) asking whether every complemented subspace of a Banach lattice must be linearly isomorphic to some Banach lattice. This question is explicitly mentioned by Casazza, Kalton, and Tzafriri at the beginning of \cite{CKT87}, where the authors refer to it as ``\textit{one of the most important problems in the theory of Banach lattices, which is still open}". Although this article dates from 1987, the problem must have been well known to specialists for some time. 

One of the reasons why we can perhaps not expect to find this problem stated much earlier in the literature is that it is by no means trivial to determine whether a Banach space can be linearly isomorphic to a Banach lattice. As far as we are concerned, the first examples of Banach spaces not isomorphic to Banach lattices were found during the 1970s and, to this end, two notions of local unconditional structure, which generalize that of $\mathcal{L}_p$-space, -- DPR-lust introduced in \cite{DPR} and GL-lust, in \cite{GL74}-- have proven useful. Among the examples that fail these properties (and thus cannot be isomorphic to Banach lattices), we find the space of compact operators $\mathcal{K}(\ell_2)$ \cite[Theorem 5.1]{GL74}, the space of bounded holomorphic functions on the disk $\mathcal{H}^\infty(\mathbb{D})$ \cite{Pelczynski}, James' space $\mathcal{J}$ \cite[Theorem 8]{Lacey-lust}, the Kalton–Peck space \cite{JLS80}... 

Both notions of local unconditional structure are closely related to that of complemented subspaces in a Banach lattice. In fact, it is known that a Banach space $E$ has GL-lust if and only if its bidual $E^{**}$ is complemented in a Banach lattice. In particular, every complemented subspace of a Banach lattice has GL-lust. However, it is still unknown whether every complemented subspace of a Banach lattice has DPR-lust. In fact, the \textit{main conjecture} of the 1975 paper \cite{FJT75} --see page 396-- is that the latter is true, so probably the CSP has been an important motivation for studying the DPR-lust. It is not even known whether GL-lust and DPR-lust are equivalent, although in this direction it is worth mentioning that a Banach space $X$ has GL-lust if and only if $X\oplus c_0$ has DPR-lust \cite[second Corollary, p. 49]{Lacey-lust}. For more information on local unconditional structures, we refer the reader to \cite[Chapter 17]{DJT}, \cite[Section 34]{TJ-book} and \cite{Lacey-lust}.

Apart from having local unconditional structure, Banach lattices have certain properties which are not shared by general  Banach spaces. These include the following (for a Banach lattice $X$):
\begin{itemize}
    \item $X$ is reflexive if and only if it does not contain subspaces isomorphic to $c_0$ or $\ell_1$. This is no longer true for Banach spaces, as James' space \cite{J51} or the second family of Bourgain–Delbaen spaces in \cite{BD} show.
    \item $X$ is weakly sequentially complete if and only no subspace of $X$ is isomorphic to $c_0$. James' space also fails this property.
    \item $X$ contains an unconditional basic sequence. However, there exist Banach spaces which are hereditarily indecomposable \cite{GM} and hence these cannot contain unconditional basic sequences.  
\end{itemize}
The proof of the above results can be found in \cite[Section 1.c]{LT2-book} and, in fact, all of them are also valid when $X$ is a complemented subspace of a Banach lattice. Several other isomorphic properties of Banach lattices (and, often, their complemented subspaces) can be found in \cite{BL76, BVL, FJT75, Ghoussoub, Johnson, Lacey-lust, T72}.

The CSP has relevant variants for $L_1$-spaces and $C(K)$-spaces. In his 1960 paper, Pe{\l{}}czy{\'{n}}ski poses the following two questions \cite[p. 210]{P60}:
\begin{enumerate}
    \item[$\mathbf{P_1}.$] Is every complemented subspace of a $C(K)$-space isomorphic to a $C(K)$-space?
    \item[$\mathbf{P_2}.$] Is every complemented subspace of an $L_1$-space isomorphic to an $L_1$-space?
\end{enumerate}
Separable versions of the above problems have been intensively studied. More specifically, it is conjectured that every complemented subspace of $C[0,1]$ is isomorphic to a $C(K)$-space and that every complemented subspace of $L_1[0,1]$ is isomorphic to $\ell_1$ or to $L_1[0,1]$ (see the dicussion at the end of Chapter 5 of \cite{AK-book}).

While Problem $\mathbf{P_2}$ is still open, Problem $\mathbf{P_1}$ has recently been answered in the negative by Plebanek and Salguero-Alarcón \cite{PS2}: a $1$-complemented subspace --denoted by $\PS$-- of a $C(K)$-space is exhibited which is not isomorphic to any $C(K)$-space. Shortly thereafter, Martínez-Cervantes, Salguero-Alarcón and the two authors of this paper have proved in \cite{dHMST23} that $\PS$ is, in fact, not isomorphic to any Banach lattice, showing that CSP (for Banach lattices) also has a negative solution. Moreover, it is also proven in \cite{dHMST23} that a slight modification of $\PS$ gives a counterexample to the CSP for complex Banach lattices. In Section \ref{sec:PS2}, we will briefly sketch this construction. It should be noted that $\PS$ is a non-separable space, so the CSP is still open both for separable $C(K)$-spaces and Banach lattices.

Along the text, several open questions related to (and motivated by) this problem will also be presented. To study these questions and, in general, to better understand the complemented subspaces of Banach lattices we will make use of the relatively recent tool of \textit{free Banach lattices generated by Banach spaces}, which was introduced by Avilés, Rodríguez and Tradacete in \cite{ART18}. These objects provide a \textit{canonical place} to investigate this problem. This is because if a Banach space $E$ is complemented in some Banach lattice, then it must be complemented in $\fbl[E]$, the free Banach lattice generated by $E$, and notably with the best possible projection constant (Proposition \ref{prop:FBL-complemented}). Further connections between free Banach lattices and the CSP are established throughout Section \ref{sec:FBLs}. It is worth highlighting that free Banach lattices actually provide a new criterion for distinguishing between Banach lattices and their complemented subspaces (Proposition \ref{prop:characterization-isomorphic-Banach-lattices}).

The fact that the projection constant obtained in the aforementioned Proposition \ref{prop:FBL-complemented} is optimal may be relevant, since the \textit{Norm-one Complemented Subspace Problem} has many peculiarities that we attempt to illustrate in Section \ref{sec:constant}. In Section \ref{sec:extra properties}, we address another kind of question: suppose that $E$ is a Banach space which satisfies a certain property ($P$), can we find a Banach lattice $Y$ having the property ($P$) which contains a complemented subspace isomorphic to $E$? We will survey several positive results in this direction and present new additions.

Section \ref{sec:hyperplanes} is devoted to analyzing a very particular case of the Complemented Subspace Problem which we call \textit{The hyperplane problem for Banach lattices}: is every hyperplane of a Banach lattice isomorphic to a Banach lattice? Although in many cases it is very easy to answer this question in the affirmative, the problem is open in general. Finally, in Section \ref{sec:more questions}, we will present and discuss further relevant open questions concerning the complemented subspaces of Banach lattices.

%To be isomorphic to an $L_1$-space is a three space property \cite[Theorem 3.4.b]{CG-book}. However, the three space problem for $C(K)$-spaces has a negative answer and this was observed by F. Cabello (see \cite[Theorem 3.5.b]{CG-book}).  Being isomorphic to a BL is not a 3SP, see, for instance \cite{KP79} (in \cite{JLS80} it is shown that this space fails GL-lust). \textcolor{blue}{Cuidado aquí. Hay distintas nociones de ser propiedad de tres espacios, y creo que en este libro en general la idea es ver si se preservan propiedades POR SUMAS TORCIDAS. Lo de Talagrand parece diferente} 

\section{A complemented subspace of a Banach lattice which is not isomorphic to a Banach lattice}\label{sec:PS2}

In 2023, Plebanek and Salguero-Alarcón provided the first example of a complemented subspace of a $C(K)$-space which is not isomorphic to a $C(K)$-space \cite{PS2}, solving in the negative a longstanding and intensively studied open question in the theory of $C(K)$-spaces. We refer the reader to \cite[Section 5]{Ros} for a summary of the most relevant results concerning the CSP for $C(K)$-spaces. The space constructed in \cite{PS2} will be denoted by $\PS$. It has been recently shown that $\PS$ cannot be linearly isomorphic to any Banach lattice \cite[Theorem 4.4]{dHMST23}, so this example has thus also served to solve the Complemented Subspace Problem for Banach lattices in the negative. Let us briefly recall the nature of this Banach space.

\subsection{Description of $\PS$} Let $\text{fin}(\mathbb N)$ denote the set of finite subsets of $\mathbb N$. Recall that a family $\mathcal{A}$ of infinite subsets of $\mathbb{N}$ is said to be \textit{almost disjoint} if $A\cap B\in \text{fin}(\mathbb N)$ for every distinct $A, B\in\mathcal{A}$. For an almost disjoint family $\mathcal{A}$, we define
$$
\text{JL}(\mathcal{A}):=\overline{\text{span}}\bigl\{\mathbf{1}_A\::\:A\in \text{fin}(\mathbb{N})\cup\mathcal{A}\cup\{\mathbb{N}\}\bigr\}\subset \ell_\infty,
$$
and call this space the Johnson-Lindenstrauss space associated to $\mathcal{A}$, as it was first introduced in \cite[Example 2]{JL1974}. Since $\text{JL}(\mathcal{A})$ is a closed sublattice of $\ell_\infty$ containing the constant function $\mathbf{1}$, by Kakutani's representation theorem for AM-spaces it follows that it is lattice isometric to a $C(K)$-space \cite[Theorem 1.b.6]{LT2-book}. It is not difficult to check that $\text{JL}(\mathcal{A})^*$ is linearly isometric to $\ell_1(\mathbb{N}\cup\mathcal{A}\cup\{\mathbb{N}\})$, so that $K$ must be scattered. This particular $C(K)$-space construction method has produced various examples with exotic properties, including:
\begin{itemize}
    \item The first example of two Banach spaces which are Lipschitz isomorphic but not linearly isomorphic \cite{AL}.
    \item A Banach space $X$ that cannot be linearly isometric to any Banach lattice, but such that $X^\mathcal{U}$ is linearly isometric to $c_0^\mathcal{U}$ for some ultrafilter $\mathcal{U}$ \cite{HHM83}.
    \item A $C_0(K)$-space admitting only few operators and decompositions \cite{KL}.
    \item The above-mentioned counterexample to the CSP for $C(K)$-spaces \cite{PS2}.
\end{itemize}

Let $\mathcal{A}=\{A_\xi\::\:\xi<\mathfrak{c}\}$ be an almost disjoint family of $\mathbb{N}$ of cardinality $\mathfrak{c}$. We write $\widehat{\mathbb{N}}=\mathbb{N}\times\{0,1\}$ and for every $\xi<\mathfrak{c}$, $n\in\mathbb{N}$, denote $\widehat{A_\xi}=A_\xi\times\{0,1\}$ and $c_n=\{(n,0),(n,1)\}$. For $\xi<\mathfrak{c}$, we decompose $\widehat{A_\xi}=B_\xi^0\cup B_\xi^1$ in such a way that for all $n\in\mathbb{N}$, the sets $B_\xi^0\cap c_n$ and $B_\xi^1\cap c_n$ are singletons. Note that $\mathcal{B}:=\{B_\xi^0,B_\xi^1\::\:\xi<\mathfrak{c}\}$ is an almost disjoint family of $\widehat{\mathbb{N}}$. With a slight abuse of notation, we will denote by $\text{JL}(\mathcal{A})$ the closed subspace of $\ell_\infty(\widehat{\mathbb{N}})$ spanned by $\{\mathbf{1}_{c_n}\::\:n\in\mathbb{N}\}\cup\bigl\{\mathbf{1}_{\widehat{A_\xi}}\::\:\xi<\mathfrak{c}\bigr\}\cup\bigl\{\mathbf{1}_{\widehat{\mathbb{N}}}\bigr\}$. Similarly, we define $\text{JL}(\mathcal{B}):=\overline{\text{span}}\bigl\{\mathbf{1}_B\::\:B\in \text{fin}(\widehat{\mathbb{N}})\cup\mathcal{B}\cup\{\widehat{\mathbb{N}}\}\bigr\}\subset \ell_\infty(\widehat{\mathbb{N}})$. 

It should be observed that $\text{JL}(\mathcal{A})$ is precisely the subspace of $\text{JL}(\mathcal{B})$ consisting of all functions which are constant on each fiber $c_n$ and the map 
$P: \text{JL}(\mathcal{B}) \to \text{JL}(\mathcal{B})$ defined as 
\[ Pf(n,0) = Pf(n,1) = \frac{1}{2}\big(f(n,0)+f(n,1)\big), \quad n\in\mathbb{N},\]
is a norm-one projection whose image is $\text{JL}(\mathcal{A})$. Now, we write $X=\text{ker}\,P$, so we have $\text{JL}(\mathcal{B})=\text{JL}(\mathcal{A})\oplus X$. Note that $X$ is also a $1$-complemented subspace of $\text{JL}(\mathcal{B})$ since the projection $Q=\text{id}_{\text{JL}(\mathcal{B})}-P$ is given by
$$
Qf(n,0) = - Qf(n,1) = \frac12\big(f(n,0) - f(n,1)\big),\quad n\in\mathbb{N}.
$$
In \cite{PS2}, two almost disjoint families, $\mathcal{A}$ and $\mathcal{B}$, are constructed in the form we just described in such a way $X$ is not isomorphic to a $C(K)$-space.

\subsection{$\PS$ is not isomorphic to a Banach lattice} We now briefly sketch how it is established in \cite{dHMST23} that $\PS$ cannot even be isomorphic to a Banach lattice. Local theory of Banach spaces plays a key role in reducing this problem. Let us begin by recalling the notion of $\mathcal{L}_p$-space:

Given $1\leq p\leq \infty$ and $\lambda\geq 1$, a Banach space $X$ is said to be an \emph{$\mathcal{L}_{p,\lambda}$-space} if for every finite-dimensional subspace $E$ of $X$ there is a finite-dimensional subspace $F$ of $X$ such that $E\subset F$ and $F$ is $\lambda$-isomorphic to $\ell_p^{n}$, where $n=\text{dim}\, F$. We say that a Banach space $X$ is an \emph{$\mathcal{L}_p$-space} if it is an $\mathcal{L}_{p,\lambda}$-space for some $\lambda$. The most relevant properties of these classes of Banach spaces may be found in \cite{LP68, LR69}.

Let us also recall two important classes of Banach lattices. A Banach lattice $X$ is said to be an \textit{AL-space} if $\|x+y\|=\|x\|+\|y\|$ for all $x,y\in X$ with $x,y\geq 0$. On the other hand, $X$ is called an \textit{AM-space} if it satisfies $\|x\lor y\|=\max\{\|x\|,\,\|y\|\}$ for every pair of positive vectors $x,y\in X$. These two notions are dual to each other: $X$ is an AL-space --respectively, an AM-space-- if and only if $X^*$ is an AM-space --resp., an AL-space-- \cite[Theorem 7, Section 3]{Lacey-book}. Kakutani's representation theorems allow us to identify (in a lattice isometric way) AL-spaces with $L_1$-spaces \cite[Theorem 1.b.2]{LT2-book} and AM-spaces with sublattices of $C(K)$-spaces \cite[Theorem 1.b.6]{LT2-book}.

Since $\PS$ is an isomorphic predual of $\ell_1(\Gamma)$, then it is an $\mathcal{L}_\infty$-space \cite[Theorem III (a)]{LR69}. Therefore, for this space, being isomorphic to a Banach lattice is equivalent to being isomorphic to an AM-space \cite[Corollary 2.2]{dHMST23}. Furthermore, since $\PS$ has a countable norming set (as it is a subspace of $\ell_\infty(\widehat{\mathbb{N}})$), being isomorphic to a Banach lattice is in this case equivalent to \textit{being isomorphic to a sublattice of $\ell_\infty$} \cite[Proposition 4.1]{dHMST23}.

The condition that $\PS$ is \textit{not} isomorphic to a sublattice of $\ell_\infty$ can be rewritten as follows: for every norming sequence $(e_n^*)_{n=1}^\infty \subset B_{\PS^*}$ there exists $f\in \PS$ such that there is no $g\in \PS$ for which $|e_n^*(f)|=e_n^*(g)$ for all $n\in\mathbb{N}$. In other words, for every norming sequence in $B_{\PS^*}$ we can always find a function in $\PS$ which does not have \textit{modulus} with respect to that sequence. A careful analysis of the construction of $\PS$ (explained in detail in \cite[Theorem 4.4]{dHMST23}) reveals that for every norming sequence in $B_{\PS^*}$ there exists $\xi<\mathfrak{c}$ such that $1_{B_\xi^0}-1_{B_\xi^1}\in \PS$ has no modulus with respect to that sequence, and hence $\PS$ cannot be isomorphic to a Banach lattice.

We want to point out that the space $\PS$ can be slightly modified to also provide a negative answer to \textit{the CSP for complex Banach lattices} \cite[Section 5]{dHMST23}. This is noteworthy, as results for 1-complemented subspaces of complex Banach lattices often differ significantly from those in the real setting \cite{DKL, Kalton(Dales), KW}.

\section{Complemented subspaces of Banach lattices via free Banach lattices}\label{sec:FBLs}

Let us start by recalling that \textit{the free Banach lattice generated by a Banach space $E$} is a pair ($\fbl[E]$, $\delta_E$), where $\fbl[E]$ is a Banach lattice and $\delta_E:E\to \fbl[E]$ is a linear isometric embedding, which has the following universal property: for every operator $T:E\to X$ into a Banach lattice $X$, there exists a unique lattice homomorphism $\widehat{T}:\fbl[E]\to X$ which extends $T$, in the sense that $\widehat{T}\delta_E=T$, and $\|\widehat{T}\|=\|T\|$. A useful way to visualize this situation is by means of the following commutative diagram:
$$
	\xymatrix{\fbl[E]\ar@{-->}^{\exists!\,\widehat{T}}[drr]&& \\
	E\ar@{^{(}->}[u]^{\delta_E}\ar[rr]^T&&X}
$$
It should be noticed that, by definition, if such an object exists, then it must be unique up to lattice isometry and that is the reason why we speak about \textit{the} free Banach lattice over $E$ and not \textit{a} free Banach lattice over $E$. Its existence was proven in \cite{ART18}, where in fact an explicit functional representation of the free Banach lattice over a Banach space was given. Let us briefly recall this construction: Let $E$ be a Banach space and denote by $H[E]$ the linear subspace of $\mathbb{R}^{E^*}$ consisting of all positively homogeneous functions $f:E^*\to\mathbb{R}$. Given $f\in H[E]$, define
$$
\|f\|_{\fbl[E]}=\sup\left\{\sum_{k=1}^n |f(x_k^*)|\::\:n\in\mathbb{N}, \:(x_k^*)_{k=1}^n\subset E^*,\:\sup_{x\in B_E}\sum_{k=1}^n|x_k^*(x)|\leq 1\right\},
$$
and note that the vector space $H_1[E]:=\{f\in H[E]\::\: \|f\|_{\fbl[E]}<\infty\}$ endowed with the pointwise order and the above norm is a Banach lattice. For each $x\in E$, let $\delta_x:E^*\to\mathbb{R}$ be defined by $\delta_x(x^*):=x^*(x)$, $x^*\in E^*$ and observe that $\delta_x\in H_1[E]$ and $\|x\|=\|\delta_x\|_{\fbl[E]}$. It was shown in \cite{ART18} that the free Banach lattice generated by $E$ is the Banach lattice $\fbl[E]=\overline{\text{lat}}\{\delta_x\::\:x\in E\}\subset H_1[E]$ together with the linear isometric embedding $\delta_E:E\to\fbl[E]$ given by $\delta_E(x):=\delta_x$.

In \cite{JLTTT}, this notion was extended to certain \textit{subcategories} of the category of Banach lattices and lattice homorphisms, such as, for example, the category of $p$-convex Banach lattices and lattice homomorphisms. Recall that, given $1\leq p\leq \infty$, a Banach lattice $X$ is said to be \emph{$p$-convex} if there is a constant $M\geq 1$ such that for every choice of vectors $(x_k)_{k=1}^n$ in $X$ we have
$$
\left\|\left(\sum_{k=1}^n|x_k|^p\right)^{\frac{1}{p}}\right\|\leq M\left(\sum_{k=1}^n\|x_k\|^p\right)^{\frac{1}{p}}, \qquad \text{if } 1\leq p<\infty,
$$
or
$$
\left\|\bigvee_{k=1}^n |x_k| \right\|\leq M \max_{1\leq k\leq n} \|x_k\|, \qquad \text{ if } p=\infty.
$$
Such expressions $\left(\sum_{k=1}^n|x_k|^p\right)^{\frac{1}{p}}$ are well defined thanks to Krivine's functional calculus (for details, see \cite[Section 1.d]{LT2-book}). The smallest possible value of $M$ is called the $\textit{$p$-convexity constant}$ of $X$ and will be denoted by $M^{(p)}(X)$. It should be noted that every Banach lattice $X$ is $1$-convex with $M^{(1)}(X)=1$. On the other hand, it can be checked that if $X$ is an $\infty$-convex Banach lattice, then it is lattice $M^{(\infty)}(X)$-isomorphic to an AM-space (see \cite[Theorem 2.1.12]{M-N-book}).

Given a Banach space $E$ and $p\in[1,\infty]$, \textit{the free $p$-convex Banach lattice over $E$} is a $p$-convex Banach lattice $\fbl^{(p)}[E]$ with $p$-convexity constant equal to $1$ together with a linear isometric embedding $\delta_E:E\to \fbl^{(p)}[E]$ with the property that for every $p$-convex Banach lattice $X$ and every operator $T:E\to X$, there exists a unique lattice homomorphism $\widehat{T}:\fbl^{(p)}[E]\to X$ such that $\widehat{T}\delta_E=T$ and, moreover, $\|\widehat{T}\|\leq M^{(p)}(X)\|T\|$. Note that for $p=1$ this definition coincides with the one of free Banach lattice generated by a Banach space introduced at the beginning of the section, that is, for any Banach space $E$, $\fbl[E]$ and $\fbl^{(1)}[E]$ coincide. Throughout the text, we will only use the notation $\fbl[E]$. 

The existence of the $\fbl^{(p)}[E]$ for any $1\leq p\leq \infty$ was proven in \cite{JLTTT}, where the authors also give a functional representation of this object in the same spirit as that of $\fbl[E]$. In this case, they consider again the space $H[E]\subset \mathbb{R}^{E^*}$ of positively homogeneous functions and define  
\begin{equation*}
\|f\|_{\fbl^{(p)}[E]}=\sup\left\{\left(\sum_{k=1}^n |f(x_k^*)|^p\right)^\frac{1}{p}\::\:n\in\mathbb{N}, \:(x_k^*)_{k=1}^n\subset E^*,\:\sup_{x\in B_E}\sum_{k=1}^n|x_k^*(x)|^p\leq 1\right\},
\end{equation*}
when $p<\infty$, and
\begin{equation*}
	\|f\|_{\fbl^{(\infty)}[E]}= \sup_{x^*\in B_{E^*}} |f(x^*)|,
\end{equation*}
when $p=\infty$. It can be shown that 
$$
\fbl^{(p)}[E]:=\overline{\text{lat}\{\delta_E(x)\::\: x\in E\}}^{\|\cdot\|_{\fbl^{(p)}[E]}}\subset H_p[E]=\{f\in H[E]\::\: \|f\|_{\fbl^{(p)}[E]}<\infty\},
$$
together with the linear isometric embedding $\delta_E:E\to\fbl^{(p)}[E]$ given by $\delta_E(x)(x^*):=x^*(x)$ is the free $p$-convex Banach lattice generated by $E$. Moreover, in the case $p=\infty$, it is pointed out in \cite[Proposition 2.2]{OTTT24} that 
with this procedure we obtain that $\fbl^{(\infty)}[E]$ is precisely $C_{ph}(B_{E^*})$, the space of positively homogeneous $w^*$-continuous functions on $B_{E^*}$, equipped with the supremum norm and the pointwise order. 

For details of these constructions, see \cite{JLTTT}. We also refer the reader to \cite{OTTT24} for an extensive study of the properties of these objects.

Free Banach lattices have a close connection with the Complemented Subspace Problem for Banach lattices. They serve as a canonical object to study this problem in the sense that if a Banach space $E$ is complemented in some Banach lattice, then it must be complemented in its corresponding $\fbl[E]$. The following proposition pinpoints this idea. An analogous result for $C(K)$-spaces can be found in \cite[Lemma, p.~247]{BL}.

\begin{prop}\label{prop:FBL-complemented}
Let $E$ be a Banach space. If $E$ is $C_1$-isomorphic to a $C_2$-complemented subspace of a Banach lattice, then $\delta(E)$ is $C_1 C_2$-complemented in $\fbl[E]$ 
\end{prop}
\begin{proof}
By hypothesis, there exist a subspace $Y$ of a Banach lattice $X$, an isomorphism $T:E\to Y$ such that $\|T\|\|T^{-1}\|\leq C_1$ and a projection $P:X\to X$ onto $Y$ such that $\|P\|\leq C_2$. By the universal property of $\fbl[E]$, there is a unique lattice homomorphism $\widehat{\iota T}:\fbl[E]\to X$ which extends $\iota T$ (where $\iota$ stands for the natural inclusion of $Y$ into $X$ as a subspace) and $\|\widehat{\iota T}\|=\|\iota T\|=\|T\|$. We have the following diagram:
$$
\xymatrix{\fbl[E] \ar@{-->}[rrd]^{\widehat{\iota T}}&&&&& \\
E\ar@{^{(}->}[u]^{\delta_E} \ar[r]^T &Y\ar@{^{(}->}[r]^\iota & X \ar[r]^P& Y \ar[r]^{T^{-1}} & E\ar@{^{(}->}[r]^{\delta_E\quad}& \fbl[E]}
$$
Define $Q=\delta_E T^{-1} P \widehat{\iota T}$. Observe that $Q(\fbl[E])\subset \delta_E(E)$ and $Q^2=Q$. In addition,
$$
Q\delta_E(x)=\delta_E T^{-1} P \widehat{\iota T}\delta_E(x)=\delta_ET^{-1}P\iota T(x)=\delta_E T^{-1}T(x)=\delta_E(x), \text{ for } x\in E,
$$
so $Q(\fbl[E])=\delta_E(E)$. Finally, note that
$$
\|Q\|=\|\delta_E T^{-1} P \widehat{\iota T}\|=\|T^{-1} P \widehat{\iota T}\|\leq \|T^{-1}\|\|P\|\|\widehat{\iota T}\|=\|T^{-1}\|\|P\|\|T\|\leq C_1C_2. \eqno\qedhere
$$
\end{proof}

\begin{rem}\label{rem:complemented-duals}
Consider the category of \textit{dual Banach spaces} together with the \textit{adjoint operators} (or equivalently, $w^*$-continuous linear maps). Given a Banach space $E$, let $J_E:E\rightarrow E^{**}$ denote the canonical embedding of $E$ into its bidual. It is clear that every operator $T:E\rightarrow F^*$ can be uniquely extended to an adjoint operator $\overset{*}{T}:E^{**}\rightarrow F^*$, given by $\overset{*}{T}=(T^*\circ J_F)^*$, in such a way that the following diagram commutes:
$$
\xymatrix{E^{**}\ar@{-->}[rd]^{\exists!\;\overset{*}{T}}&\\
     E\ar[r]^{T}\ar^{J_E}[u]& F^*}
$$
Moreover, $\|\overset{*}{T}\|=\|T\|$. Thus, we can consider the pair $(E^{**},\: J_E)$ as the \textit{free dual Banach space generated by $E$}, as it was pointed out in Subsection 1.1 of \cite{OTTT24}. By mimicking the proof of the preceding proposition, it can be shown that if a Banach space $E$ is $C_1$-isomorphic to a $C_2$-complemented subspace of a dual Banach space, then $J_E(E)$ is $C_1C_2$-complemented in $E^{**}$. This is a well-known observation that can be found, for instance, in \cite[Remark, p. 16]{L64}.  
\end{rem}

In fact, using free Banach lattices again, we can identify precisely when a complemented subspace of Banach lattice is actually isomorphic to a Banach lattice:

\begin{prop}\label{prop:characterization-isomorphic-Banach-lattices}
A Banach space $E$ is linearly isomorphic to a Banach lattice if and only if there is an ideal $I$ in $\fbl[E]$ such that $\fbl[E]=\delta_E(E)\oplus I$. More precisely, if $T:E\to X$ is an isomorphism, then $\fbl[E]=\delta_E(E)\oplus \text{ker}(\widehat{T})$.
\end{prop}
\begin{proof}
Let $X$ be a Banach lattice and suppose that there exists an isomorphism $T:E\to X$. By the universal property of $\fbl[E]$, there is a lattice homomorphism $\widehat{\text{T}}:\fbl[E]\to X$ such that $\widehat{\text{T}}\circ \delta_{E}=T$. Therefore, $P=\delta_{E}T^{-1} \widehat{T}$ defines a projection on $\fbl[E]$ with range $\delta_{E}(E)$ and $\text{ker}(P)=\text{ker}(\widehat{\text{T}})$ is an ideal $I$ in $\fbl[E]$. We have that $\fbl[E]=\delta_{E}(E)\oplus I$.

To prove the converse implication, recall that the quotient of a Banach lattice over a norm-closed ideal is a Banach lattice by \cite[Corollary 1.3.14]{M-N-book}. Thus, if we have a decomposition $\fbl[E]=\delta_E(E)\oplus I$, then $\fbl[E]/I$ is a Banach lattice isomorphic to $\delta_E(E)$ and therefore also to $E$.
\end{proof}

\begin{rem}
    Just by mimicking the above proof, one may show that a Banach space $E$ is isomorphic to a $p$-convex Banach lattice if and only if there is an ideal in $\fbl^{(p)}[E]$ such that $\fbl^{(p)}[E]=\delta_E(E)\oplus I$ (for any $1\leq p\leq \infty$). In particular, $E$ is isomorphic to an AM-space if and only if $C_{ph}(B_{E^*})=\fbl^{(\infty)}[E]=\delta_E(E)\oplus I$. 
\end{rem}

Consequently, at least at a theoretical level, free Banach lattices provide a possible criterion for distinguishing between Banach lattices and their complemented subspaces. In order to be able to use Proposition \ref{prop:characterization-isomorphic-Banach-lattices}, a natural first step would be to understand how closed ideals in free Banach lattices look like. It is well known that ideals in $C(K)$-spaces are \textit{sets of zeros}, that is, given an ideal $I\subset C(K)$, there is a closed subset $F\subset K$ such that $I=\{f\in C(K)\::\: f(t)=0 \text{ for all } t\in F\}$ (\cite[Proposition 2.1.9]{M-N-book}). We will show in the next proposition that \textit{the same} holds for $\fbl^{(\infty)}[E]$. Given a Banach space $E$, we will say that a set $K\subset B_{E^*}$ is \textit{positively homogeneous} in $B_{E^*}$ if $K=\mathbb{R}_+ K \cap B_{E^*}$, that is, if $0\in K$ and if $0\neq x^*\in K$, then $\lambda x^*\in K$ for every $\lambda\in \bigl(0,\frac{1}{\|x^*\|}\bigr]$. Let us start with a simple (and probably well-known) lemma for which we have not found a precise reference, so we include its proof below.

\begin{lem}
Let $X$ be an AM-space and $I\subset X$ a (closed) ideal. Then $X/ I$ is an AM-space.
\end{lem}
\begin{proof}
Recall that by \cite[Corollary 1.3.14]{M-N-book} we already know that $X/I$ is a Banach lattice, so we simply have to check that $\|\overline{x}\lor \overline{y}\|=\max\{\|\overline{x}\|, \|\overline{y}\|\}$ for every pair of positive elements $\overline{x}, \overline{y}\in X/I$. Let us fix $\varepsilon>0$ and $\overline{x},\overline{y}\in \bigl(X/I\bigr)_+$. Then, there exist $z,w\in I$ such that $\|x+z\|\leq \|\overline{x}\|+\varepsilon$ and $\|y+w\|\leq \|\overline{y}\|+\varepsilon$. Thus, we have that
\begin{equation}\label{eq:21}
 \varepsilon+ \max\{\|\overline{x}\|, \|\overline{y}\|\}\geq \max\{\|x+z\|,\|y+w\|\}=\bigl\||x+z|\lor |y+w|\bigr\|\geq \bigl\|\overline{|x+z|}\lor \overline{|y+w|}\bigr\|.   
\end{equation}
Now, observe that %$|x+z|\lor |y+w|-|x|\lor |y|\in I$. 
%Indeed, on the one hand we obtain that
%$$|x+z|\lor |y+w| - |x|\lor |y|\leq \bigl( |x|+|z|+|w|\bigr) \lor \bigl( |y|+|w|+|z|\bigr) - |x|\lor |y|=|z|+|w|,$$
%and, on the other hand, we have that
%$$|x+z|\lor |y+w| - |x|\lor |y|\geq \bigl( |x|-|z|-|w|\bigr) \lor \bigl( |y|-|w|-|z|\bigr) -|x|\lor |y|= -|z|-|w|,$$
$\bigl||x+z|\lor |y+w|- |x|\lor |y| \bigr|\leq |z|+|w|$ and, since  $|w|$ and $|z|$ are in $I$, we conclude that $|x+z|\lor |y+w|- |x|\lor |y|\in I$. Moreover, given that $\overline{x}$ and $\overline{y}$ are positive elements, we have that $|x+z|\lor |y+w|-x\lor y \in I$. Thus, from equation (\ref{eq:21}) we get that
$$ \varepsilon+ \max\{\|\overline{x}\|, \|\overline{y}\|\}\geq \bigl\|\overline{|x+z|}\lor \overline{|y+w|}\bigr\|=\| \overline{x}\lor \overline{y}\|,$$
and since $\varepsilon>0$, $\overline{x}$ and $\overline{y}$ were arbitrarily chosen, we conclude that $X/I$ is an AM-space.
\end{proof}

\begin{prop}\label{prop:ideals-free-infinity}
Let $E$ be a Banach space. Given an ideal $I$ in $\fbl^{(\infty)}(E)=C_{ph}(B_{E^*})$, there exists a $w^*$-closed positively homogeneous subset $K$ of $B_{E^*}$ such that $I=\{f\in C_{ph}(B_{E^*})\::\: \left.f\right|_{K}=0\}$. Moreover, the operator $T:C_{ph}(B_{E^*})/I\to C_{ph}(K)$ defined by $T\bar{f}:=\left.f\right|_K$ is a (surjective) lattice isometry.
\end{prop}
\begin{proof}
Define $K:=\{x^*\in B_{E^*}\::\: f(x^*)=0 \text{ for all } f\in I\}$. Observe that by the $w^*$-continuity on $B_{E^*}$ of the elements of $I\subset C_{ph}(B_{E^*})$, $K$ is a $w^*$-closed subset of $B_{E^*}$ and also, by the positive homogeneity of these functions, if $x^*\in K$, then $\lambda\cdot x^*\in K$ for every $\lambda\in \bigl[0,\frac{1}{\|x^*\|}\bigr]$. Let us write $I_K:=\{f\in C_{ph}(B_{E^*})\::\: \left.f\right|_{K}=0\}$.

It is clear that $I\subset I_{K}$, so let us see the reverse inclusion. Suppose that $f\notin I$, hence $\bar{f}\in C_{ph}(B_{E^*})/I$ is non-zero. By the previous lemma, we know that $C_{ph}(B_{E^*})/I$ is an AM-space, so there exists a norm-one lattice homomorphism $\overline{z^*}$ on $C_{ph}(B_{E^*})/I$ such that $\overline{z^*}(\bar{f})\neq 0$ (see, for instance, \cite[Proposition 5.4]{BGHMT}). Now denote by $Q:C_{ph}(B_{E^*})\to C_{ph}(B_{E^*})/I$ the canonical quotient map $Qg:=\bar{g}$, which is a lattice homomorphism \cite[Proposition 1.3.13]{M-N-book}. Therefore, $\overline{z^*}\circ Q\in B_{C_{ph}(B_{E^*})^*}$ is a lattice homomorphism and so there is $x^*\in B_{E^*}$ such that $\widehat{x^*}=\overline{z^*}\circ Q$. Note that $x^*\in K$, because for every $g\in I$ we have
$$
g(x^*)=\widehat{x^*}(g)=\overline{z^*}(Qg)=\overline{z^*}(\bar{0})=0.
$$
On the other hand, $f(x^*)=\overline{z^*}(\bar{f})\neq 0$, so $f\notin I_K$.

For the last statement of the proposition, let us consider the mapping $T:C_{ph}(B_{E^*})/I\to C_{ph}(K)$ defined by $T\bar{f}:=\left.f\right|_K$. It is clear that $T$ is well defined (since $I=I_K$) and that it is a lattice homomorphism. Let us check its surjectivity first. Recall that, by \cite[Proposition 5.4]{BGHMT}, $C_{ph}(B_{E^*})/I$ can be identified with $C_{ph}\left(K_{C_{ph}(B_{E^*})/I}\right)$ in a lattice isometric way, where $K_{C_{ph}(B_{E^*})/I}=\text{Hom}\bigl(C_{ph}(B_{E^*})/I\bigr)\cap B_{C_{ph}(B_{E^*})/I}$. Given $f\in C_{ph}(K)$, we define $\overline{f}:K_{C_{ph}(B_{E^*})/I}\to\mathbb{R}$ by $\overline{f}(\overline{z^*}):=f(Q^*\overline{z^*})$, for $\overline{z^*}\in K_{C_{ph}(B_{E^*})/I}$. Observe that $Q^*\overline{z^*}=\overline{z^*}\circ Q$ is a lattice homomorphism on $C_{ph}(B_{E^*})$ which is zero on $I$ and hence there exists a unique $x^*\in K$ such that $\widehat{x^*}=Q\overline{z^*}$. Therefore, $\overline{f}$ is well defined and it is easy to see that it is $w^*$-continuous positively homogeneous on $K_{C_{ph}(B_{E^*})/I}$, so $\overline{f}\in C_{ph}(B_{E^*})/I$. Finally, note that by construction $T\overline{f}=f$.

It remains to show that $T$ is norm-preserving. Fix $\bar{f}\in C_{ph}(B_{E^*})/I$ with $\|\bar{f}\|=1$. Given that $C_{ph}(B_{E^*})/I$ is an AM-space, there exists a norm-one lattice homomorphism $\overline{z^*}$ on $C_{ph}(B_{E^*})/I$ such that $\overline{z^*}(\bar{f})=1$. We have shown above that there is $x^*\in K$ such that $\widehat{x^*}=Q^*\overline{z^*}$, so we have
$$
1=\|\bar{f}\|=\overline{z^*}(\bar{f})=Q^*\overline{z^*}(f)=f(x^*)\leq \sup_{y^*\in K} |f(y^*)|=\|\left.f\right|_K\|.
$$
Now, suppose that $f(x^*)>1$ for some $x^*\in K$. Since $\|\bar{f}\|=1$, there exists $g\in I$ such that $f(x^*)>\|f+g\|$. Therefore, we have
$$
f(x^*)>\|f+g\|=\sup_{y^*\in B_{E^*}} |f(y^*)+g(y^*)|\geq |f(x^*)|,
$$
which is absurd.
\end{proof}

As a consequence of the previous proposition, we can obtain the following characterization of Banach spaces isomorphic to AM-spaces. This should be compared with a very similar and already known characterization of Banach spaces isomorphic to $C(K)$-spaces (see, for instance, \cite[Lemma 2.2]{PS1}).

\begin{cor}
A Banach space $E$ is isomorphic to an AM-space if and only if there exists a $w^*$-closed subset $K$ of $B_{E^*}$ such that $K$ is norming for $E$ and for every $x\in E$, there exists $y\in E$ such that $x^*(y)=|x^*(x)|$ for every $x^*\in K$. 
\end{cor}
\begin{proof}
Suppose that $E$ is isomorphic to an AM-space $X$ and let $T:E\to X$ be an isomorphism. By the universal property of $\fbl^{(\infty)}[E]=C_{ph}(B_{E^*})$, there exists a unique lattice homomorphism $\widehat{T}:C_{ph}(B_{E^*})\to X$ such that $\widehat{T}\delta_E=T$. As $\text{ker}(\widehat{T})$ is an ideal in $C_{ph}(B_{E^*})$, by the previous proposition there is $w^*$-closed positively homogeneous subset $K$ of $B_{E^*}$ such that $\bar{f}\in C_{ph}(B_{E^*})/\text{ker}(\widehat{T})\mapsto \left.f\right|_K\in C_{ph}(K)$ defines a surjective lattice isometry. Therefore, the mapping $x\in E\mapsto \left.\delta_x\right|_K\in C_{ph}(K)$ is an isomorphism, so $K$ satisfies the desired conditions.

Conversely, let $K$ be a $w^*$-closed subset of $B_{E^*}$ which is norming for $E$ and for every $x\in E$ we can find $y\in E$ such that $x^*(y)=|x^*(x)|$ for all $x^*\in K$. Let us define $T:E\to C(K)$ by $Tx:=\delta_x$, where $\delta_x(x^*):=x^*(x)$ for $x^*\in K$. Since $K$ is norming, we deduce that $T$ is bounded below, so $T$ is an isomorphism onto its range. Moreover, note that $T(X)$ is a sublattice of $C(K)$ since, by hypothesis, for every $x\in X$, there is $y\in X$ such that $Ty=|\delta_x|$. Therefore, $T(X)$ is an AM-space.
\end{proof}

However, the description of ideals in free Banach lattices seems more complicated than in $C(K)$-spaces. We will now see that ideals in free Banach lattices need not be \textit{sets of zeros} (in the sense of Proposition \ref{prop:ideals-free-infinity}); in fact, an ideal in $\fbl[E]$ can have no zeros. To illustrate this circumstance, we will use the following simple observation:

\begin{prop}\label{prop:kernel-zero}
    Let $X$ be a Banach lattice. $x^*\in X^*$ is a lattice homomorphism if and only if $\text{ker}(\beta)\subset \text{ker}(\widehat{x^*})$, where $\beta:\fbl[X]\to X$ is the unique lattice homomorphism such that $\beta\circ\delta_X=\text{id}_X$.
\end{prop}
\begin{proof}
Suppose that $x^*:X\to\mathbb{R}$ is a lattice homomorphism. Then, $x^*\circ\beta$ and $\widehat{x^*}$ are lattice homomorphisms on $\fbl[X]$ which clearly agree on $\delta_X(X)$. By uniqueness of the extension, they actually coincide on $\fbl[X]$, so $\widehat{x^*}=x^*\circ\beta$. This shows that $\text{ker}(\beta)\subset\text{ker}(\widehat{x^*})$.

Now, assume that $\text{ker}(\beta)\subset\text{ker}(\widehat{x^*})$. Then the functional $\overline{x^*}:\fbl[X]/\text{ker}(\beta)\to \mathbb{R}$, given by $\overline{x^*}(\bar{f}):=f(x^*)$, is a well-defined lattice homomorphism on $\fbl[X]/\text{ker}(\beta)$. On the other hand, note that the mapping $\bar{\beta}:\fbl[X]/\text{ker}(\beta)\to X$ defined by $\bar{\beta}(\bar{f})=\beta f$ is a bijective lattice homomorphism and it is not difficult to determine its inverse explicitly: the composition $Q\delta_X:X\to\fbl[X]/\text{ker}(\beta)$, where $Q:\fbl[X]\to \fbl[X]/\text{ker}(\beta)$ stands for the canonical quotient defined by $Qf:=\bar{f}$. Indeed, for every $x\in X$, we have
$$
\bar{\beta}Q\delta_X(x)=\bar{\beta}(\overline{\delta_X(x)})=\beta\delta_X(x)=x.
$$
Also, for $f\in\fbl[X]$, we have
$$
Q\delta_X\bar{\beta} (\bar{f})=Q\delta_X\beta (f)=Q(f)=\bar{f},
$$
where the second-to-last identity follows from the fact that $\beta\delta_X=id_X$ and $ker(Q)=ker(\beta)$. Thus, $Q\delta_X$ is also a lattice homomorphism, so $\overline{x^*}\circ Q\delta_X$ is a lattice homomorphism on $X$. But observe that
$$
\overline{x^*}\circ Q\delta_X(x)=\overline{x^*}\bigl(\overline{\delta_X(x)}\bigr)=\delta_X(x)(x^*)=x^*(x),\quad \text{for } x\in X,
$$
so $x^*=\overline{x^*}\circ Q\delta_X$ and this shows that $x^*$ is a lattice homomorphism.
%We will see that it is also norm-preserving. Fix any $\bar{f}\in \fbl[X]/\text{ker}(\beta)$. For every $g\in \text{ker}(\beta)$ we have
%$$
%\|\bar{\beta}(\bar{f})\|=\|\beta(f+g)\|\leq \|f+g\|,
%$$
%so by taking infimum over $g\in \text{ker}(\beta)$ we infer that $\|\bar{\beta}(\bar{f})\|\leq \|\bar{f}\|$. Furthermore, it is clear that the inverse of $\bar{\beta}$ is $Q\delta_X$, where $Q:\fbl[X]\to \fbl[X]/\text{ker}(\beta)$ stands for the canonical quotient defined by $Qf:=\bar{f}$. Indeed, for every $x\in X$, $\bar{\beta}Q\delta_X(x)=\bar{\beta}(\overline{\delta_X(x)})=\beta\delta_X(x)=x$. As $\|Q\delta_X\|\leq 1$, we conclude that $\|\bar{\beta}\bar{f}\|=\|\bar{f}\|$ for all $\bar{f}\in \fbl[X]/\text{ker}(\beta)$. 
\end{proof}

Consequently, if we take a Banach lattice $X$ which does not have lattice homomorphisms (for example, $L_p[0,1]$, for any $1\leq p<\infty$), then $\text{ker}(\beta)$ is an ideal in $\fbl[X]$ which does not have zeros (in the sense that for every non-zero $x^*\in B_{X^*}$ we can find $f\in \text{ker}(\beta)$ such that $f(x^*)\neq 0$).

Comparing the previous proposition with Proposition \ref{prop:ideals-free-infinity} it is natural to wonder whether having a decomposition $\fbl[X]=\delta_X(X)\oplus \text{ker}(\beta)$, with $\text{ker}(\beta)$ being a set of zeros, implies that $X$ is an AM-space. This does not necessarily have to be the case, as the following shows.

\begin{prop}\label{prop:kernel-separate}
    Given a Banach lattice $X$,  $\text{Hom}(X,\mathbb R)$ separates the points of $X$ if and only if there exists a $w^*$-closed positively homogeneous subset $K$ of $B_{X^*}$ such that $\text{ker}(\beta)=\{f\in \fbl[X]\::\: \left.f\right|_{K}=0\}$. In this situation $K=\text{Hom}(X,\mathbb{R})\cap B_{X^*}$.
\end{prop}
\begin{proof}
    Let us assume that $\text{Hom}(X,\mathbb{R})$ separates the points of $X$ and define $K:=\text{Hom}(X,\mathbb{R})\cap B_{X^*}$, which is a $w^*$-closed positively homogeneous subset of $B_{X^*}$, and $Z_K:=\{f\in \fbl[X]\::\: \left.f\right|_K=0\}$. By the preceding proposition, $\text{ker}(\beta)\subset Z_K$. Now, take $f\in Z_K$. Since $f\in \fbl[X]$, there exists a unique decomposition $f=\delta_X(x)+g$, for some $x\in X$ and $g\in\text{ker}(\beta)$. Then, $\delta_X(x)\in Z_K$ and since $K$ separates the points of $X$ we deduce that $x=0$. Hence $f=g\in \text{ker}(\beta)$.
    
    For the reverse implication, suppose that $\text{ker}(\beta)=Z_K$, for some $w^*$-compact positively homogeneous subset $K$ of $B_{X^*}$. Proposition \ref{prop:kernel-zero} guarantees that each element $x^*\in K$ is a lattice homomorphism on $X$. Since $\delta_X(X)\cap Z_K=\{0\}$, given $x\neq 0$, there is $x^*\in K$ such that $x^*(x)=\delta_x(x^*)\neq 0$. Therefore, $K$ separates the points of $X$ and, in particular, so does $\text{Hom}(X,\mathbb R)$.    
\end{proof}

The next result gathers \textit{an isomorphic version} of Propositions \ref{prop:kernel-zero} and \ref{prop:kernel-separate}. Before proving this, let us recall the definition of a prominent class of lattice homomorphisms between free Banach lattices. Given two Banach spaces $E$ and $F$, the universal property of $\fbl[E]$ ensures that for every operator $T:E\to F$ there is a unique lattice homomorphism $\overline{T}:\fbl[E]\to\fbl[F]$ which makes the following diagram commutative
$$
	\xymatrix{\fbl [E]\ar@{-->}[rr]^{\overline T}&&\fbl[F]\\
	E\ar[u]_{\delta_E}\ar^{T}[rr]&& F\ar[u]_{\delta_F}}
$$
that is, $\overline{T}\delta_E=\delta_FT$. It is easy to check that $\overline{T}$ is a lattice isomorphism if and only $T$ is an isomorphism and also that this extension is given by $\overline{T}f=f\circ T^*$, for $f\in \fbl[E]$ \cite[Lemma 3.1]{OTTT24}. We refer the reader to
\cite[Section 3]{OTTT24} for more properties concerning these operators.

\begin{prop}\label{prop:isomorphic-enumeration-BLs}
Let $E$ be a Banach space and $X$ a Banach lattice. If $T:E\to X$ is an isomorphism, then $\overline{T}:\fbl[E]\to \fbl[X]$ is a lattice isomorphism (the only one satisfying $\overline{T}\delta_E=\delta_X T$), which is given by $\overline{T}f=f\circ T^*$, for $f\in \fbl[E]$. Let $\beta:\fbl[X]\to X$ and $\widehat{T}:\fbl[E]\to X$ be the unique lattice homomorphisms satisfying $\beta\delta_X=\text{id}_X$ and $\widehat{T}\delta_E=T$, respectively. The following assertions hold:
\begin{enumerate}
\item $\overline{T}\bigl(\text{ker}(\widehat{T})\bigr)=\text{ker}(\beta)$;
\item $\overline{T}\bigl(\text{ker}(\widehat{T^*x^*})\bigr)=\text{ker}(\widehat{x^*})$ for any $x^*\in X^*$;
\item $x^*\in \text{Hom}(X,\mathbb{R})$ if and only if $\text{ker}(\widehat{T})\subset \text{ker}(\widehat{T^*x^*})$;
    \item $\text{Hom}(X,\mathbb{R})$ separates the points of $X$ if and only if there exists a $w^*$-closed positively homogeneous subset $K$ of $B_{E^*}$ such that $\text{ker}(\widehat{T})=Z_K$.   
\end{enumerate} 
\end{prop}
\begin{proof} ($1$) The composition $\beta \overline{T}:\fbl[E]\to X$ is a lattice homomorphism such that for every $x\in E$ we have
$
\beta \overline{T}\bigl(\delta_E(x)\bigr)=\beta \delta_X (Tx)=Tx=\widehat{T}\bigl(\delta_E(x)\bigr)
$. Since the sublattice generated by $\{\delta_E(x)\::\:x\in E\}$ is dense in $\fbl[E]$, then $\beta \overline{T}=\widehat{T}$, so $\overline{T}\bigl(\text{ker}(\widehat{T})\bigr)\subset\text{ker}(\beta)$. For the other inclusion, observe that the latter implies that $\beta=\widehat{T}\overline{T}^{-1}$.

\noindent ($2$) This can be checked in a similar way to the previous claim using on this occasion the identity of lattice homomorphisms $\widehat{x^*}\overline{T}=\widehat{T^*x^*}$.

\noindent ($3$)  By Proposition \ref{prop:kernel-zero} we know that $x^*\in X^*$ is a lattice homomorphism if and only if $\text{ker}(\beta)\subset \text{ker}(\widehat{x^*})$. And it can be easily derived from ($1$) and ($2$) that the last-mentioned condition is equivalent to $\text{ker}(\widehat{T})\subset \text{ker}(\widehat{T^*x^*})$.

\noindent ($4$) Proposition \ref{prop:kernel-separate} states that $\text{Hom}(X,\mathbb{R})$ separates the points of $X$ is equivalent to the fact that $\text{ker}(\beta)=\{f\in\fbl[X]\::\: f(x^*)=0 \text{ for all } x^*\in\text{Hom}(X,\mathbb{R})\}$. So if $\text{Hom}(X,\mathbb{R})$ separates the points of $X$, then by ($1$) we have
$$
\text{ker}(\widehat{T})=\overline{T}^{-1}\bigl(\text{ker}(\beta)\bigr)=\left\{\overline{T}^{-1}f\::\: \overline{T}^{-1}f(T^*x^*)=0 \text{ for all } x^*\in\text{Hom}(X,\mathbb{R})\right\},
$$
and this shows that $\text{ker}(\widehat{T})=Z_K$ for $K=T^*\bigl(\text{Hom}(X,\mathbb{R})\bigr)\cap B_{E^*}$. Conversely, suppose that there is a $w^*$-closed positively homogeneous subset $K$ of $B_{E^*}$ such that $\text{ker}(\widehat{T})=Z_K$. Therefore, for any $z^*\in K$, $\text{ker}(\widehat{T})\subset\text{ker}(\widehat{z^*})$. Then, using ($1$) and ($2$), we obtain $\text{ker}(\beta)=\overline{T}\bigl(\text{ker}(\widehat{T})\bigr)\subset \overline{T}\bigl(\text{ker}(\widehat{z^*})\bigr)=\text{ker}(\widehat{(T^{-1})^*z^*})$, so by Proposition \ref{prop:kernel-zero} we conclude that $(T^{-1})^*(K)\subset \text{Hom}(X,\mathbb{R})$. Moreover, $(T^{-1})^*(K)$ separates the points of $X$. Indeed, given a non-zero $x\in X$, as $\delta_E(E)\cap Z_K=\{0\}$, there exists $z^*\in K$ such that
$$
0\neq \delta_E(T^{-1}x)(z^*)=\overline{T^{-1}}\delta_X(x)(z^*)=\delta_X(x)\bigl((T^{-1})^*z^*\bigr)=(T^{-1})^*z^*(x).
$$
Therefore, $\text{Hom}(X,\mathbb{R})$ separates the points of $X$.
\end{proof}

We have seen in the fourth assertion of the preceding result that being isomorphic to a Banach lattice whose set of lattice homomorphisms separates its points implies the existence of a $w^*$-compact $K\subset B_{E^*}$ such that $\fbl[E]=\delta_E(E)\oplus Z_K$. We will show now that the converse also holds.

\begin{prop}\label{prop:isomorphic-zeros}
    Let $E$ be a Banach space and suppose that there is a $w^*$-closed subset $K$ of $B_{E^*}$ such that $\fbl[E]=\delta_E(E)\oplus Z_K$. Then $E$ is isomorphic to a Banach lattice $X$ such that $\text{Hom}(X,\mathbb{R})$ separates the points of $X$.
\end{prop}
\begin{proof}
First, note that $Z_K=\{f\in \fbl[E]\::\: f(x^*)=0 \text{ for all } x^*\in K\}$ is a closed ideal in $\fbl[E]$. Thus, by \cite[Corollary 1.3.14]{M-N-book}, $X:=\fbl[E]/Z_K$ is a Banach lattice. For every $x^*\in K$, we define $\overline{x^*}:X\to\mathbb{R}$ by $\overline{x^*}(\overline{f}):=f(x^*)$, which is a well-defined lattice homomorphism, so $\{\overline{x^*}\::\: x^*\in K\}\subset \text{Hom}(X,\mathbb{R})$. We will show that $\{\overline{x^*}\::\: x^*\in K\}$ separates the points of $X$. Take any non-zero $\bar{f}\in \fbl[E]/Z_K$. As $\fbl[E]=\delta_E(E)\oplus Z_K$, there exists $0\neq x\in E$ such that $\bar{f}=\overline{\delta_E(x)}$. In particular, given that $\delta_E(E)\cap Z_K=\{0\}$ there must be $x^*\in K$ such that $x^*(x)\neq 0$. As a result, we have
$$
\overline{x^*}(\bar{f})=\overline{x^*}\bigl(\overline{\delta_E(x)}\bigr)=\delta_E(x)(x^*)=x^*(x)\neq 0,
$$
and this shows that $\{\overline{x^*}\::\: x^*\in K\}\subset \text{Hom}(X,\mathbb{R})$ separates the points of $X$.
\end{proof}

Given an isomorphism $T:E\to X$ from a Banach space $E$ onto a Banach lattice $X$, we know by Proposition \ref{prop:characterization-isomorphic-Banach-lattices} that we have a decomposition $\fbl[E]=\delta_E(E)\oplus \text{ker}(\widehat{T})$. The ideal $\text{ker}(\widehat{T})$ is not a set of zeros unless $\text{Hom}(X,\mathbb{R})$ separates the points of $X$ (by Proposition \ref{prop:isomorphic-enumeration-BLs}). However, one could still wonder whether the following holds: Is there always a $w^*$-compact $K\subset B_{E^*}$ such that $\fbl[E]=\delta_E(E)\oplus Z_K$? If the latter were true, this would imply by the preceding proposition that every Banach lattice is isomorphic to a Banach lattice whose set of lattice homomorphisms separates its points. The next example shows that this is not always the case.

\begin{example}
 Let us consider $L_1[0,1]$ and suppose that there exists a Banach lattice $X$ which is isomorphic to $L_1[0,1]$ and such that $\text{Hom}(X,\mathbb{R})$ separates its points. By \cite[Proposition 2.1]{dHMST23}, $X$ is lattice isomorphic to an $L_1$-space, so we can directly assume that $X$ is an $L_1$-space. By \cite[Corollary to Theorem 9, Section 14]{Lacey-book}, $X$ must be lattice isometric to $\ell_1$, $L_1[0,1]$ or $L_1[0,1]\oplus_1\ell_1(\Gamma)$, where $|\Gamma|\leq\aleph_0$. Since we are assuming that the lattice homomorphisms of $X$ separates its points, then $X=\ell_1$. But $L_1[0,1]$ is not isomorphic to $\ell_1$, so we have arrive at a contradiction. It should be noted that this argument cannot be extended to $L_p[0,1]$ for $1<p<\infty$, since it is well known that for these spaces the Haar basis is unconditional \cite[Theorem 6.1.7]{AK-book}.
\end{example}

We conclude this section by pointing out that all the results shown here can be easily adapted to the complex setting, using the notion of free complex Banach lattice introduced in \cite{dHT23}. 

\section{The relevance of projection constants}\label{sec:constant}

Proposition \ref{prop:FBL-complemented} brings an additional peculiarity that we have not discussed so far: if $E$ is a $C$-complemented subspace of some Banach lattice (i.e., $E$ is the range of a projection of norm $C$), then $\delta_E(E)$ is complemented in $\fbl[E]$ \textit{with constant less than or equal to $C$}. The reason for taking into account the projection constant is that there are significant differences between what happens for the \textit{contractive case} ($C=1$) and the general case in some of the most relevant classes of Banach lattices. Let us recall some well-known results in this direction:
\begin{itemize}
    \item \textit{For any $1\leq p<\infty$, every $1$-complemented subspace of an $L_p$-space is isometric to an $L_p$-space.} This was proven by Tzafriri in 1969 in \cite{T69} (see also \cite{BL74}), extending to general measure spaces previous results for $L_p$-spaces over probability spaces due to Douglas \cite{Douglas} (case $p=1$) and Ando \cite{Ando} (cases $1<p\neq 2<\infty$). This result was generalized to Hilbert-valued $L_p$-spaces by Raynaud \cite{Raynaud}. We refer the reader to the extensive survey by Randrianantoanina \cite{R01} for more information on $1$-complementation in Köthe function spaces and sequence spaces.

    If we do not assume that the projections have norm $1$, the situation changes significantly. In this regard, we should mention that Bourgain, Rosenthal and Schechtman showed the existence of uncountably many mutually non-isomorphic complemented subspaces of $L_p[0,1]$, for any $1<p\neq 2<\infty$ \cite{BRS}. In contrast, for $p=1$, as we already mentioned, it is conjectured that the only possible complemented subspaces (up to isomorphism) are $\ell_1$ and $L_1[0,1]$ \cite[Conjecture 5.7.7]{AK-book}.

    \item \textit{Every separable $1$-complemented subspace of a $C(K)$-space is isomorphic to a $C(K)$-space.} This is a consequence of the following two facts: in \cite[Theorem 3 (i)]{LW69} it was shown that every $1$-complemented subspace of a $C(K)$-space is linearly isometric to some $C_\sigma(K)$-space; shortly after, Samuel \cite{Samuel} proved that separable $C_\sigma(K)$-spaces are isomorphic to $C(K)$-spaces (see also \cite[Lemma 5]{B73}). In 1973, Benyamini generalized the preceding result: every separable $G$-space is isomorphic to a $C(K)$-space \cite{B73}. Recall that a G-space is exactly a $1$-complemented subspace in an AM-space (up to a linear isometry) \cite[Theorem 3 (ii)]{LW69}. In the non-separable setting, this result is no longer true: $\PS$ is a $1$-complemented subspace of a $C(K)$-space which is not even isomorphic to a Banach lattice \cite{dHMST23, PS2}.

    For not necessarily contractive projections, the following \textit{conjecture} deserves to be mentioned again: Is every complemented subspace of $C[0,1]$ linearly isomorphic to $C(K)$? \cite[Conjecture 5.7.8]{AK-book}.
 
    \item \textit{For complex scalars, it was proven by Kalton and Wood in 1976 \cite{KW} that every $1$-complemented subspace of a Banach space with a $1$-unconditional basis must have a $1$-unconditional basis} (see also \cite{Flinn, Rosenthal-KW}). This theorem does not hold in the real case (see \cite{BFL} or the last example of \cite{L79}). However, the more general question of whether every complemented subspace of a space with an unconditional basis has an unconditional basis is still open in both real and complex cases \cite[Problem 1.d.5]{LT1-book}.
\end{itemize}

Since all norms on a finite-dimensional vector space are equivalent, every finite-dimensional Banach space is trivially isomorphic to a Banach lattice. In particular, finite-dimensional Banach spaces are complemented subspaces of Banach lattices. A naive question would be whether these spaces are complemented by some uniform constant (not depending on the dimension). Well-known examples show this is not the case: $n$-dimensional Schatten $p$-class operators $S_p^n$ for $p\neq 2$ \cite[Theorem 5.1]{GL74}, appropriate finite-dimensional subspaces of James space \cite{JT} (see also \cite[Theorem 34.3]{TJ-book}). Next result provides a similar argument making use of free Banach lattices.

\begin{prop}
For every $C\geq 1$, there exists a finite-dimensional Banach space which is not $C$-complemented in any Banach lattice.
\end{prop}
\begin{proof}
Let us consider the James space $\mathcal{J}$ (see \cite{J51} or \cite[Section 3.4]{AK-book}). The reason why we are interested in taking this particular Banach space is that it has the following three properties:
\begin{enumerate}
    \item $\mathcal{J}$ has a (monotone) basis.
    \item $\mathcal{J}$ is isometric to its bidual $\mathcal{J}^{**}$.
    \item $\mathcal{J}$ cannot be isomorphic to any complemented subspace of a Banach lattice.
\end{enumerate}

We will make use of $\mathcal{J}^{**}$ to construct a sequence of finite-dimensional spaces which cannot be uniformly complemented in their corresponding free Banach lattices. Note that $\mathcal{J}^{**}$ also satisfies the properties $(1)-(3)$. Let $(e_n)_{n=1}^\infty$ be a monotone basis of $\mathcal{J}^{**}$ with associated basis projections $(P_n)_{n=1}^\infty$ and let $E_n=\text{span}\{e_k\::\: 1\leq k\leq n\}$ denote the range of $P_n$.

 Suppose that there exists a constant $C\geq 1$ such that for every natural $n$ there is a projection $Q_n:\fbl[E_n]\to\fbl[E_n]$ with range $\delta_{E_n}(E_n)$ such that $\|Q_n\|\leq C$. For every $n\in\mathbb{N}$, $\overline{P_n}:\fbl[\mathcal{J}^{**}]\to \fbl[\mathcal{J}^{**}]$ defines a projection whose range is $\overline{\iota_n}(\fbl[E_n])$ (where $\iota_n$ stands for the canonical inclusion of $E_n$ into $\mathcal{J}^{**}$ as a subspace). Observe that $\overline{\iota_n}:\fbl[E_n]\to \fbl[\mathcal{J}^{**}]$ is a lattice isometric embedding given that $E_n$ is $1$-complemented in $\mathcal{J}^{**}$ \cite[Theorem 3.7]{OTTT24}. We will denote by $\widetilde{Q_n}$ the projection on $\overline{\iota_n}(\fbl[E_n])$ defined by $\widetilde{Q_n}(\overline{\iota_n}f)=\overline{\iota_n}(Q_n f)$, for every $n\in\mathbb{N}$ and every $f\in \fbl[E_n]$.

For every natural $n$, let $R_n:=\widetilde{Q_n}\overline{P_n}$, which is a projection on $\fbl[\mathcal{J}^{**}]$ onto $\overline{\iota_n}\bigl(\delta_{E_n}(E_n)\bigr)=\delta_{\mathcal{J}^{**}}(E_n)\subset \delta_{\mathcal{J}^{**}}(\mathcal{J}^{**})$. Let $\mathcal{U}$ be a free ultrafilter on $\mathbb N$ and define
$$Rf:=\delta_{\mathcal{J}^{**}}\left(w^*-\lim_{\mathcal{U}} \delta_{\mathcal{J}^{**}}^{-1} (R_n\,f)\right),\qquad f\in \fbl[\mathcal{J}^{**}].$$

Observe that $R$ is a projection (with $\|R\|\leq C$) from $\fbl[\mathcal{J}^{**}]$
onto $\delta_{\mathcal{J}^{**}}(\mathcal{J}^{**})$, which is a contradiction with the property ($3$) mentioned above. Therefore, by Proposition \ref{prop:FBL-complemented}, this argument shows that for every $C\geq 1$ there exists $n\in\mathbb{N}$ such that $E_n$ cannot be $C$-complemented in any Banach lattice.
\end{proof}

%\subsection{Almost contractively complemented spaces}
This fact is also relevant when comparing to the Banach space situation: Given a Banach space $E$ and a complemented subspace $F$, there is an equivalent norm in $E$ for which $F$ is contractively complemented.

We say that a Banach space $E$ is \textit{almost contractively complemented in a Banach lattice} if for every $\varepsilon>0$, there is a Banach lattice $X_\varepsilon$ such that $E$ is $(1+\varepsilon)$-isomorphic to a $(1+\varepsilon)$-complemented subspace of $X_\varepsilon$.

\begin{prop}
Suppose that $E$ is almost contractively complemented in a Banach
lattice and there is a contractive projection $Q:E^{**}\to E$. Then E is contractively complemented in a Banach lattice. 
\end{prop}
\begin{proof}
 By Proposition \ref{prop:FBL-complemented}, for every $n\in \mathbb N$ there is a projection $P_n$ on $\fbl[E]$ with range $\delta(E)$ and $\|P_n\|\leq 1+\frac{1}{n}$. Thus, $P_n^{**}$ are projections on $\fbl[E]^{**}$ such that $\|P_n^{**}\|\leq 1+\frac{1}{n}$ and with ranges isometric to $E^{**}$ (namely, all projections have the same range $\delta_E^{**}(E^{**})$). Let $\mathcal{U}$ be a free ultrafilter on $\mathbb N$ and define
 $$
 Pf=w^*-\lim_{\mathcal{U}} P_n^{**} f, \qquad \text{ for every } f\in \fbl[E]^{**}.
 $$
It is not difficult to check that $P$ is a contractive projection whose range is $\delta_E^{**}(E^{**})$. This shows that $E^{**}$ is $1$-complemented in $\fbl[E]^{**}$ and since $E$ is $1$-complemented in $E^{**}$, then $E$ is $1$-complemented in $\fbl[E]^{**}$.
\end{proof}

We do not know whether the hypothesis that $E$ is contractively complemented in $E^{**}$ in the above proposition is actually necessary. This should be compared with the fact that if a Banach space $E$ is \textit{almost contractively complemented in a $C(K)$-space}, then it is actually isometric to a $C_\sigma(K)$-space \cite[Theorem 0.2]{AB88}.

%\subsection{Some specific projections} $L_p$-projections, $M$-projections? \textcolor{blue}{Esto lo comenté ya un poco en el artículo del contraejemplo del CSP y no quiero solaparme... Se puede motivar un poco más el interés de estas proyecciones y decir que hay artículos importantes estudiándolas (F. Cunningham Jr.). L-structure in L-spaces, M-structure in Banach spaces, M-structure in dual Banach spaces, Lp-structure in real Banach spaces. También el paper de Vidal Agniel sobre Lp-proyecciones.}

\section{Complementation in Banach lattices with extra properties}\label{sec:extra properties}

If a Banach space which is complemented in a Banach lattice has a certain property, it is sometimes possible to construct a Banach lattice with this extra property in which it also embedds as a complemented subspace. For instance, if a separable Banach space $E$ is complemented in a Banach lattice, then it is complemented in a separable Banach lattice; specifically, in its free Banach lattice $\fbl[E]$ (this follows from the fact that the sublattice generated by a subset of a Banach lattice preserves the density character). In a similar direction, recall  \cite[Proposition 1.c.6]{LT2-book} (see also \cite[Proposition 2.6 (i)]{FJT75}):

\begin{prop}\label{prop:not-containing-c0}
    Let $E$ be a Banach space which does not contain isomorphic copies of $c_0$ (respectively, $E$ does not contain $(\ell_\infty^n)_{n=1}^\infty$ uniformly) and is a complemented subspace of a Banach lattice $X$. Then, there is a Banach lattice $Y$ which does not contain isomorphic copies of $c_0$ (resp., $(\ell_\infty^n)_{n=1}^\infty$ uniformly) and contains $E$ as a complemented subspace.
\end{prop}

The following is also well-known. We include a proof for the convenience of the reader.

\begin{prop}\label{prop:complemented-reflexive}
    If $E$ is a reflexive Banach space which is complemented in a Banach lattice $X$, then $E$ is a complemented subspace of some reflexive Banach lattice.
\end{prop}
\begin{proof}
Let $P:X\to X$ be a projection onto $E$. It should be noted that $P$ is a weakly compact operator since its range $P(X)=E$ is a reflexive Banach space. As a consequence of \cite[Corollary 2.7]{AB84}, we know that $P^2=P$ factors through a reflexive Banach lattice $Y$. That is, there exist operators $T:X\to Y$ and $S:Y\to X$ such that $P=ST$. Now, define $Q:=TPS$, which is an operator on $Y$. Observe that $\left.T\right|_E$ is an isomorphism into its image and $Q$ is a projection on $Y$ (this can be deduced straightforwardly from the identity $P=ST$) with range $T(E)$. To check the last fact, note that for every $x\in E=P(X)$ we have
$$
Q(Tx)=TPS(Tx)=TP(ST)(x)=TPP(x)=Tx.
$$
\end{proof}

A \textit{dual version} of Proposition \ref{prop:not-containing-c0} was also established in \cite[Theorem 1.2]{FGJ}.
\begin{prop}\label{prop:complemented-dual}
    Let $E$ be a complemented subspace of a Banach lattice $X$ and assume that $c_0$ does not embed into $E^*$. Then $E$ is complemented in a Banach lattice $Y$ such that $c_0$ does not embed into $Y^*$.
\end{prop}

Recall that by Proposition \ref{prop:FBL-complemented} we know that if a Banach space $E$ is a complemented subspace of some Banach lattice $X$, then $E$ must be complemented in $\fbl[E]$. So it is natural to wonder whether we can take $Y=\fbl[E]$ in the three previous propositions. Let us analyze this:
\begin{itemize}
    \item $\fbl[E]$ contains an isomorphic of $c_0$ whenever $E$ is a Banach space such that $\text{dim}\,E\geq 2$. Indeed, let $E$ be a Banach space of dimension $\geq 2$ and let $F$ be a $2$-dimensional subspace of $E$. Then $\fbl[F]$ is a complemented sublattice of $\fbl[E]$ and, moreover, $\fbl[F]$ is $2$-lattice isomorphic to $C(S_{F^*})\approx C[0,1]$ (see \cite[Remark 3.1 (i)]{O24}). Thus, $c_0$ embeds isomorphically into $\fbl[F]$, and hence into $\fbl[E]$. Therefore, we cannot deduce Propositions \ref{prop:not-containing-c0} and \ref{prop:complemented-reflexive} using free Banach lattices (at least, not in a trivial way).
    
    \item However, in \cite[Theorem 9.20]{OTTT24} it is shown that $\ell_1$ is a complemented subspace of $E$ if and only if $\ell_1$ is a complemented subspace of $\fbl[E]$. Moreover, recall that by Bessaga-Pe\l{}czy\'{n}ski's theorem \cite[Proposition 2.e.8]{LT1-book}, for any Banach space $F$, $\ell_1$ embeds complementably into $F$ if and only if $c_0$ embeds isomorphically into $F^*$. Consequently, we can take $Y=\fbl[E]$ in Proposition \ref{prop:complemented-dual}.
\end{itemize}
As a result of Proposition \ref{prop:FBL-complemented} and \cite[Corollary 9.25 and Lemma 9.26]{OTTT24}, we can also state a \textit{local version} of the previous proposition:

\begin{prop}
    Let $E$ be a complemented subspace of a Banach lattice $X$ and assume that $E^*$ does not contain $(\ell_\infty^n)_{n=1}^\infty$ uniformly. Then $E$ is complemented in a Banach lattice $Y$ such that $Y^*$ does not contain uniformly subspaces isomorphic to $\ell_\infty^n$. Namely, we can take $Y=\fbl[E]$. 
\end{prop}

%{\bf\color{red} Deberías incluir al menos las definiciones y propiedades necesarias de espacios script L1 y Linfty}

We now turn to analyze the case when $E$ is an $\mathcal{L}_\infty$-space. When $E$ is an $\mathcal{L}_1$-space, we have the following result:

\begin{prop}
    If $E$ is an $\mathcal{L}_1$-space which is complemented in a Banach lattice, then it is complemented in some $L_1(\mu)$-space.
\end{prop}
\begin{proof}
As $E$ is an $\mathcal{L}_1$-space, then thanks to \cite[Proposition 7.1]{LP68} we know that it is isomorphic to a subspace of an $L_1(\mu)$-space. Since every $L_1(\mu)$-space is weakly sequentially complete and this property passes to subspaces, $E$ cannot contain isomorphic copies of $c_0$. By \cite[Proposition 1.c.6]{LT2-book} (see \cite[Proposition 2.6]{FJT75}) $E$ is a complemented subspace of a certain Banach lattice $X$ which does not contain isomorphic copies of $c_0$. By \cite[Theorem 1.c.4]{LT2-book}, the canonical image of $X$ in $X^{**}$ is a projection band of $X^{**}$. Thus, by Remark \ref{rem:complemented-duals}, $E$ is complemented in its bidual, and by \cite[Corollary 1 of Theorem 7.1]{LP68} we conclude that $E$ is a complemented subspace of an $L_1(\mu)$-space.
\end{proof}

\begin{rem}
The above proof actually shows that if $E$ is an $\mathcal{L}_1$-space the following assertions are equivalent:
\begin{enumerate}
    \item[(i)] $E$ is complemented in a Banach lattice;
    \item[(ii)] $E$ is complemented in its bidual;
    \item[(iii)] $E$ is complemented in an $L_1$-space.
\end{enumerate}
Not every $\mathcal{L}_1$-space satisfies the above conditions (see the examples $D_k$ constructed in \cite[p. 211]{LT-notes}).
\end{rem}

Can we establish an analogous version of the preceding result for $\mathcal{L}_\infty$-spaces? Recall that an $\mathcal{L}_\infty$-space which is isomorphic to a Banach lattice must be isomorphic to an AM-space \cite[Corollary 2.2]{dHMST23}, so it is natural to pose the following:

\begin{question}\label{question:complemented-script-l-infinity}
 If $E$ is an $\mathcal{L}_\infty$-space which is complemented in a Banach lattice, is it then complemented in some AM-space?    
\end{question}

Note that in the separable setting, \textit{AM-space} can be replaced by \textit{C(K)-space} \cite{B73} (while in the non-separable case this cannot be done, see \cite{B77}). One of the motivations behind this question is the construction due to Benyamini and Lindenstrauss of an isometric predual of $\ell_1$ which cannot be complemented in any $C(K)$-space \cite{BL}. Since this space is separable, it cannot be isomorphic to a Banach lattice (given that this is equivalent to being isomorphic to a $C(K)$-space for separable spaces). But could this space be complemented in a Banach lattice? If Question \ref{question:complemented-script-l-infinity} had an affirmative answer, then the answer to the latter would be negative. 

\begin{question}[Separable CSP]
Let $E$ be a \textit{separable} Banach space which is complemented in a Banach lattice $X$. Is $E$ isomorphic to a Banach lattice?    
\end{question}

But a priori, Benyamini-Lindenstrauss' example could provide a negative answer to the CSP for separable Banach lattices. The separable version of CSP is closely related to understanding the complemented subspaces of $L_1[0,1]$ and $C[0,1]$, as noted in \cite[Remark 2.5]{dHMST23}.

By \cite[Theorem 9.21]{OTTT24}, if $E$ is an $\mathcal{L}_\infty$-space, then $\fbl[E]$ satisfies \textit{an upper $2$-estimate}. In fact, we will see next that in this case $\fbl[E]$ is actually 2-convex.

\begin{lem}
    Let $X$ be a $2$-convex Banach lattice. Then, $\fbl[X]$ is $2$-convex with $M^{(2)}(\fbl[X])\leq K_G M^{(2)}(X)$.    
\end{lem}
\begin{proof}
By \cite[Proposition 9.38]{OTTT24}, the following two assertions are equivalent for any $C\geq 1$:
\begin{enumerate}
    \item $\fbl^{(2)}[X]$ is lattice $C$-isomorphic to $\fbl[X]$.
    \item Every contraction $T:X\to L_1(\mu)$ is $2$-convex with constant $C$.
\end{enumerate}
Let us check ($2$). Let $T:X\to L_1(\mu)$ be an operator such that $\|T\|\leq 1$ and let $(x_k)_{k=1}^n$ be an arbitrary finite sequence in $X$. Then, by \cite[Theorem 1.f.14]{LT2-book}, we have
$$
\left\|\left(\sum_{k=1}^n|Tx_k|^2\right)^\frac{1}{2}\right\|\leq K_G\left\|\left(\sum_{k=1}^n|x_k|^2\right)^\frac{1}{2}\right\|\leq K_G M^{(2)}(X) \left(\sum_{k=1}^n \|x_k\|^2\right)^\frac{1}{2},
$$
which shows that $T$ is $2$-convex with constant $\leq K_G M^{(2)}(X)$.
\end{proof}

\begin{rem}
Conversely, it should also be noticed that if $X$ is a Banach lattice such that $\fbl[X]$ is $2$-convex, then $X$ is also $2$-convex with $M^{(2)}(X)\leq M^{(2)}(\fbl[X])$. Indeed, given an arbitrary finite sequence $(x_k)_{k=1}^n$ in $X$, we have
\begin{eqnarray*}
\left\|\left(\sum_{k=1}^n|x_k|^2\right)^\frac{1}{2}\right\| &=&  \left\|\left(\sum_{k=1}^n|\beta\delta_X(x_k)|^2\right)^\frac{1}{2}\right\|\overset{(*)}{=}\left\|\beta\left(\sum_{k=1}^n|\delta_X(x_k)|^2\right)^\frac{1}{2}\right\| \\
&\leq& \left\|\left(\sum_{k=1}^n|\delta_X(x_k)|^2\right)^\frac{1}{2}\right\|\leq M^{(2)}(\fbl[X]) \left(\sum_{k=1}^n \|\delta_X(x_k)\|^2\right)^\frac{1}{2} \\
&=& M^{(2)}(\fbl[X]) \left(\sum_{k=1}^n \|x_k\|^2\right)^\frac{1}{2},
\end{eqnarray*}
where $\beta:\fbl[X]\to X$ is the unique lattice homomorphism such that $\beta \delta_X=\text{id}_X$. The equality $(*)$ is a consequence of Krivine's functional calculus (see, for instance, \cite[Lemma 2.1]{JLTTT}).
\end{rem}

\begin{comment}
\begin{lem}
For every $n\in\mathbb{N}$, $\fbl[\ell_\infty^n]$ is $2$-convex and $M^{(2)}(\fbl[\ell_\infty^n])\leq K_G$.    
\end{lem}
\begin{proof}
Fix any $n\in\mathbb{N}$ and let $C\geq 1$. By \cite[Proposition 9.38]{OTTT24}, the following two assertions are equivalent:
\begin{enumerate}
    \item $\fbl^{(2)}[\ell_\infty^n]$ is lattice $C$-isomorphic to $\fbl[\ell_\infty^n]$.
    \item Every contraction $T:\ell_\infty^n\to L_1(\mu)$ is $2$-convex with constant $C$.
\end{enumerate}
Let us check ($2$). Let $T:\ell_\infty^n\to L_1(\mu)$ be an operator such that $\|T\|\leq 1$ and let $(x_k)_{k=1}^m$ be an arbitrary finite sequence in $\ell_\infty^n$. Then, by \cite[Theorem 1.f.14]{LT2-book}, we have
$$
\left\|\left(\sum_{k=1}^m |Tx_k|^2\right)^\frac{1}{2}\right\|_1\leq K_G\left\|\left(\sum_{k=1}^m |x_k|^2\right)^\frac{1}{2}\right\|_\infty\leq K_G \left(\sum_{k=1}^n \|x_k\|_\infty^2\right)^\frac{1}{2},
$$
which shows that $T$ is $2$-convex with constant $\leq K_G$.
\end{proof}
\end{comment}

\begin{prop}\label{prop:fbl[linfty]2convex}
If $E$ is an $\mathcal{L}_p$-space for some $2\leq p\leq\infty$, then $\fbl[E]$ is $2$-convex.    
\end{prop}
\begin{proof}
 First, recall that as $E$ is an $\mathcal{L}_p$-space, \cite[Theorem III (c)]{LR69} ensures the existence of a constant $\rho\geq 1$ such that for every finite-dimensional subspace $G$ of $E$ there is a finite-dimensional subspace $F$ of $E$ such that $F\supset G$, $d(F,\ell_p^{\text{dim\,F}})\leq \rho$, and such that there is a projection of norm $\leq \rho$ from $E$ onto $F$. We will show that for every $(f_k)_{k=1}^n\subset \fbl[E]$,
 $$
 \left\|\left(\sum_{k=1}^n |f_k|^2 \right)^\frac{1}{2} \right\|_{\fbl[E]}\leq K_G\,\rho^2\left(\sum_{k=1}^n \|f_k\|_{\fbl[E]}^2\right)^{\frac{1}{2}}.
 $$
Since $\fvl[E]$ is norm-dense in $\fbl[E]$, we can find $x_1,\ldots, x_m\in E$ with the property that for every $k=1,\ldots,n$ there exists $g_k\in \text{lat}\{\delta_{x_1},\ldots,\delta_{x_m}\}$ such that $\|f_k-g_k\|_{\fbl[E]}\leq \frac{\varepsilon}{2K_G\rho^2 n}$. Define $G=\text{span}\{x_j\::\: j=1,\ldots, m\}\subset E$ and let $F$ be such that $F\supset G$, $d(F,\ell_p^{\text{dim\,F}})\leq \rho$ and there is a projection $P:E\to F$ ($P\iota=\text{id}_F$, where $\iota:F\hookrightarrow E$) such that $\|P\|\leq \rho$. Let us make some observations:
\begin{enumerate}
    \item For every $(a_k)_{k=1}^n,\,(b_k)_{k=1}^n\subset \mathbb{R}^n$ we have $$
    \bigl|\left\|(a_k)_{k=1}^n\right\|_2-\left\|(b_k)_{k=1}^n\right\|_2\bigr|\leq \bigl\|(a_k-b_k)_{k=1}^n\bigr\|_2\leq \bigl\|(a_k-b_k)_{k=1}^n\bigr\|_1,
    $$
    so by Krivine's functional calculus \cite[Theorem 1.d.1]{LT2-book}, we get:
    $$
     \left\|\left(\sum_{k=1}^n |f_k|^2 \right)^\frac{1}{2} -\left(\sum_{k=1}^n |g_k|^2 \right)^\frac{1}{2} \right\|_{\fbl[E]}\leq \left\|\sum_{k=1}^n |f_k-g_k|\right\|_{\fbl[E]}\leq \frac{\varepsilon}{2}
     $$
     \item Given $g\in \overline{\iota}(\fbl[F])$, we have $\overline{\iota}\overline{P}g=g$, so  $\|\overline{\iota}\overline{P}g\|_{\fbl[E]}\leq \|\overline{P}g\|_{\fbl[F]}$. Also, by \cite[Lemma 2.1]{JLTTT}, we have $\overline{P}\left(\sum_{k=1}^n |g_k|^2 \right)^\frac{1}{2}=\left(\sum_{k=1}^n |\overline{P}g_k|^2 \right)^\frac{1}{2}$.
     \item Since $F$ is $\rho$-isomorphic to $\ell_p^{\text{dim}\,F}$ (and $2\leq p\leq\infty$) then by the previous lemma we deduce that $\fbl[F]$ is $2$-convex with constant $M^{(2)}(F)\leq K_G\rho$.
\end{enumerate}

With the above comments in mind, we deduce that

\begin{eqnarray*}
 \left\|\left(\sum_{k=1}^n |f_k|^2 \right)^\frac{1}{2} \right\|_{\fbl[E]} & \overset{(1)}{\leq}&  \left\|\left(\sum_{k=1}^n |g_k|^2 \right)^\frac{1}{2} \right\|_{\fbl[E]}+\frac{\varepsilon}{2}    \overset{(2)}{\leq} \left\|\left(\sum_{k=1}^n |\overline{P}g_k|^2 \right)^\frac{1}{2} \right\|_{\fbl[F]}+\frac{\varepsilon}{2} \\
 &\overset{(3)}{\leq}& K_G\rho\left(\sum_{k=1}^n \|\overline{P}g_k\|_{\fbl[F]}^2\right)^\frac{1}{2}+\frac{\varepsilon}{2}\leq  K_G\rho^2\left(\sum_{k=1}^n \|g_k\|_{\fbl[E]}^2\right)^\frac{1}{2}+\frac{\varepsilon}{2} \\
 &\leq & K_G\rho^2\left(\sum_{k=1}^n \|f_k\|_{\fbl[E]}^2\right)^\frac{1}{2}+\varepsilon,
\end{eqnarray*}
and since $\varepsilon>0$ is arbitrary, we obtain the desired inequality.  
\end{proof}
In particular, it follows that if an $\mathcal{L}_\infty$-space is complemented in a Banach lattice, then it is complemented in a 2-convex Banach lattice. However, note that by \cite[Proposition 9.30]{OTTT24}, $\fbl[E]$ is at most $2$-convex, so the convexity in Proposition \ref{prop:fbl[linfty]2convex} cannot be improved.

\begin{comment}

\end{comment}

\section{Hyperplanes in Banach lattices}\label{sec:hyperplanes}

A noteworthy particular case of the CSP is \textit{the Hyperplane Problem for Banach lattices}:

\begin{question}
Let $X$ be a Banach lattice. Is every hyperplane of $X$ linearly isomorphic to a Banach lattice?    
\end{question}

Note that since all hyperplanes of a Banach space are mutually isomorphic, in order to answer the above question it will be enough to find a certain hyperplane which is isomorphic (or not) to a Banach lattice. Also note that if the above question has a positive answer, then every finite codimensional subspace of a Banach lattice would be linearly isomorphic to a Banach lattice.

In the particular case when a Banach lattice $X$ has a non-trivial lattice homomorphism $x^*\in X^*$, then $\text{ker}(x^*)$ is a hyperplane which is also an ideal in $X$. Thus $\text{ker}(x^*)$ is a Banach lattice with the Banach lattice structure inherited from $X$ and this shows that, in this case, the hyperplanes of $X$ are isomorphic to Banach lattices. Let us look closer at some general instances of this situation:
\begin{enumerate}
    \item If $X$ has a $1$-unconditional basis $(u_n)_{n=1}^\infty$, with biorthogonal functionals $(u_n^*)_{n=1}^\infty$, then $\text{Hom}(X,\mathbb{R})=\{\lambda u_n^*\::\:\lambda\geq 0,\:n\in\mathbb{N}\}$. Note that $\text{ker}(u_{n_0}^*)$ also has $1$-unconditional basis (for any $n_0\in\mathbb{N}$), so hyperplanes do have an unconditional basis in this case. We would like to emphasize that whether hyperplanes are isomorphic or not to the original space is not the matter here; we simply want to know whether they can carry a Banach lattice structure. So our problem is different from the classical one considered in \cite{Gowers-hyperplane} and, in fact, Gowers' counterexample has an unconditional basis, so its hyperplanes are isomorphic to Banach lattices.

    \item Recall that AM-spaces can be characterized as those Banach lattices whose set of lattice homomorphisms is $1$-norming (for instance, see \cite[Proposition 5.4]{BGHMT}). In particular, given a non-trivial AM-space $X$, we can find a non-zero $x^*\in\text{Hom}(X,\mathbb{R})$, and so $\text{ker}(x^*)$ is also an AM-space.
    
    \item From the preceding comment we know that hyperplanes in a $C(K)$-space are isomorphic to AM-spaces. But in this case, more can be said: every hyperplane of a $C(K)$-space is isomorphic to a $C(K)$-space. Indeed, take $t_1,t_2\in K$, $t_1\neq t_2$, and consider
    $$
    X=\bigl\{f\in C(K)\::\: f(t_1)=f(t_2)\bigr\}=\bigl\{f\in C(K)\::\: (\delta_{t_1}-\delta_{t_2})(f)=0\bigr\}.
    $$
    Therefore, $X$ is the kernel of $\delta_{t_1}-\delta_{t_2}\in C(K)^*$, so it is a hyperplane of $C(K)$. Moreover, it is clear that $X$ is a closed sublattice of $C(K)$ such that $\mathbf{1}_K\in X$, so by Kakutani's representation theorem for AM-spaces $X$ is lattice isometric to a $C(K)$-space.
    %We can say explicitly what the compact is: the resulting quotient topology of identifying two different points.
    On the other hand, recall that there exist $C(K)$-spaces whose hyperplanes are not isomorphic to the whole space (see, for instance, the first examples due to Koszmider \cite{Koszmider} --assuming \textbf{\textsf{CH}}-- and Plebanek \cite{Plebanek} --accomplished in ZFC), but we will not look at this question in this note.
\end{enumerate}

At the moment we have \textit{essentially} one example of a complemented subspace of a Banach lattice which is not isomorphic to a Banach lattice, namely $\textsf{PS}_2$ \cite{dHMST23}. Can this space be isomorphic to a hyperplane of some Banach lattice? We will see that this is not the case.

\begin{rem}
Suppose that $\textsf{PS}_2\oplus_\infty\mathbb{R}$ were isomorphic to a Banach lattice.  Since $\textsf{PS}_2$ is a complemented subspace of a $C(K)$-space, it is an $\mathcal{L}_\infty$-space. Therefore, $\textsf{PS}_2\oplus_\infty\mathbb{R}$ is also an $\mathcal{L}_\infty$-space, so by \cite[Corollary 2.2]{dHMST23} we may assume that it is isomorphic to an AM-space. But, we have seen in ($2$) that hyperplanes of an AM-space are Banach lattices, so $\textsf{PS}_2$ would be isomorphic to a Banach lattice, and this is a contradiction.
\end{rem}

\begin{rem}
Although on many occasions it is not hard to check that the hyperplanes of a certain Banach lattice are isomorphic to Banach lattices, the Hyperplane Problem for Banach lattices is still an open question in both the separable and non-separable cases. It should be noted that in the \textit{separable} setting it is sufficient to study this problem \textit{only for reflexive Banach lattices}. Indeed, suppose that $X$ is a separable Banach lattice. We distinguish two cases:
\begin{itemize}
    \item If $X$ contains a complemented copy of $c_0$ or $\ell_1$, then, since $c_0\approx c_0\oplus\mathbb{R}$ and $\ell_1\approx\ell_1\oplus\mathbb{R}$, we deduce that $X$ is isomorphic to its hyperplanes. In particular, its hyperplanes are isomorphic to Banach lattices.
    \item If $X$ has no complemented copies of $c_0$ and no complemented copies of $\ell_1$, we will see below that this implies that $X$ must be reflexive:
    \begin{itemize}
        \item Since $X$ is separable, not having complemented copies of $c_0$ implies not containing isomorphic copies of $c_0$. Hence, by \cite[Theorem 2.4.12]{M-N-book}, $X$ is a KB-space.
        \item Moreover, as $X$ has no complemented copies of $\ell_1$, then by a Bessaga-Pe\l{}czy\'{n}ski's result \cite[Proposition 2.e.8]{LT1-book}, $X^*$ does not contain isomorphic copies of $c_0$. Thus, using again \cite[Theorem 2.4.12]{M-N-book}, we infer that $X^*$ is a KB-space.
        \item Now, given that $X$ and $X^*$ are KB-spaces, \cite[Theorem 2.4.15]{M-N-book} ensures that $X$ is reflexive.
    \end{itemize}
\end{itemize}
\end{rem}

It is well known (see \cite[Theorem 5.59]{AA-book}) that given a Banach lattice $X$ and a positive projection $P:X\to X$, then its range $E=P(X)$ endowed with the inherited order of $X$ becomes a Banach lattice with the following:
\begin{itemize}
    \item its lattice operations are given by $x\lor_E y=P(x\lor y)$, $x\land_E y=P(x\land y)$ and $|x|_E=P|x|$;
    \item the norm $|||\cdot|||$ defined by $|||x|||=\||x|_E\|=\|P|x|\|$.
\end{itemize}

In particular, the above shows that the range of a positive projection is isomorphic to a Banach lattice. This prompts the following question: given a Banach lattice $X$, is there always a positive projection $P:X\to X$ onto one of its hyperplanes? An affirmative answer to the latter would imply a positive solution to the hyperplane problem. However, this will not be the case in general, as we will see below. To prove this we will rely on the following simple observation:
\begin{lem}
Let $X$ be a Banach lattice and let $x_0$ be a non-zero positive element of $X$. If $x_0$ is not an atom, then there exists $0\leq y_0\leq x_0$ such that neither $y_0$ nor $x_0-y_0$ can dominate a positive multiple of $x_0$. 
\end{lem}
\begin{proof}%[Proof of the Claim]\renewcommand{\qedsymbol}{(Claim) \ensuremath{\Box}}
Since $x_0$ is not an atom in $X$, there must exist a non-proportional vector $x\in X$ to $x_0$ such that $0\leq x\leq x_0$. Now, define
$$
\lambda_0:=\sup\{\lambda>0\::\: \lambda x\leq x_0\}.
$$
Observe that this supremum exists given that $X$ is a Banach lattice and so its norm is monotone on $X_+$. Moreover, this supremum is actually a maximum, as the set $X_+$ is closed. Consider $y:=x_0-\lambda_0 x$. This new vector satisfies that $0\leq y\leq x_0$ and
is not proportional to $x_0$ either, but it has an additional feature: $y$ cannot dominate a positive multiple of $x_0$. Indeed, suppose that there exists $\lambda>0$ such that 
$$
x_0-\lambda_0x=y\geq \lambda x_0.
$$
Since $x_0\geq y$ and $y$ is not proportional to $x_0$, then $\lambda< 1$. Therefore, the above inequality is equivalent to $x_0\geq \frac{\lambda_0}{1-\lambda}x$ and this contradicts the definition of $\lambda_0$. Now, define $\mu_0:=\sup\{\lambda>0\::\: \lambda y\leq x_0\}$ and $y_0:=\mu_0 y=\mu_0(x_0-\lambda_0x)$. Then we have the following: both $y_0,x_0-y_0\in [0,x_0]$ cannot dominate a positive multiple of $x_0$. 
\end{proof}

\begin{prop}\label{prop:hyperplanepositiveprojection}
Let $X$ be a Banach lattice. If there exists a hyperplane in $X$ which is complemented by a positive projection, then $X$ has an atom.   
\end{prop}

\begin{proof}
Let $P:X\to X$ be a positive projection whose range $P(X)$ is a hyperplane of $X$. Then, there exist $x_0\in X$, $x_0^*\in X^*$ such that $x_0^*(x_0)=1$ and
\begin{equation}\label{eq:positive-hyperplane}
 Px=x-x_0^*(x)x_0\geq 0, \qquad \text{ for every } x\in X_+.   
\end{equation}
By decomposing $x_0$ into its corresponding positive and negative parts, we get
\begin{equation*}
    1=x^*_0(x_0)=x_0^*(x_0^+)-x_0^*(x_0^-),
\end{equation*}
so $x_0^*(x_0^+)\geq \frac{1}{2}$ or $x_0^*(x_0^-)\leq -\frac{1}{2}$. Without lost of generality, we may assume that $x_0^*(x_0^+)\geq \frac{1}{2}$; indeed, if not, we can replace $x_0^*$ and $x_0$ with $-x_0^*$ and $-x_0$ respectively and then we are in the desired situation. If $x_0^+$ is not an atom, then by the previous lemma we can find $0\leq y_0\leq x_0^+$ such that neither $y_0$ nor $x_0^+-y_0$ can dominate a positive multiple of $x_0^+$. As we have $x_0^+=y_0+(x_0^+-y_0)$, then $x_0^*(y_0)\geq \frac{1}{4}$ or $x_0^*(x_0^+-y_0)\geq \frac{1}{4}$. We may assume that $x_0^*(y_0)\geq \frac{1}{4}$, but if we are in the other case we can argue similarly. If we evaluate expression \eqref{eq:positive-hyperplane} at $y_0$, we obtain
$$
y_0\geq x_0^*(y_0)x_0 \quad \Longleftrightarrow \quad y_0+x_0^*(y_0)x_0^-\geq x_0^*(y_0)x_0^+\geq \frac{1}{4}x_0^+.
$$
Taking infima on both sides of the last inequality with respect to $x_0^+$ and using \cite[Theorem 1.1.1 (ix)]{M-N-book}, we get that $y_0\geq \frac{1}{4}x_0^+$, and this is a contradiction with the fact that $y_0$ does not dominate any positive multiple of $x_0^+$.
\end{proof}

We close this section with a characterization of those Banach spaces whose hyperplanes are isomorphic to Banach lattices using free Banach lattices.

\begin{prop}
Let $E$ be a Banach space. The following assertions are equivalent:
\begin{enumerate}
    \item $E$ is linearly isomorphic to a Banach lattice whose hyperplanes are isomorphic to Banach lattices.
    \item $E$ is linearly isomorphic to a Banach lattice $X$ such that $\text{Hom}(X,\mathbb{R})\neq\{0\}$.
    \item There is $x^*\in E^*\backslash\{0\}$ and there is an ideal $I$ in $\fbl[E]$ such that $\fbl[E]=\delta_E(E)\oplus I$ with $I\subset \text{ker}\bigl(\widehat{x^*}\bigr)$.
    \item For every $x^*\in E^*$ there is an ideal $I$ in $\fbl[E]$ such that $\fbl[E]=\delta_E(E)\oplus I$ with $I\subset \text{ker}\bigl(\widehat{x^*}\bigr)$.
\end{enumerate}
\end{prop}
\begin{proof}
($1$) $\Rightarrow$ ($2$)  By hypothesis, $E$ is isomorphic to $Y\oplus \mathbb{R}$, where $Y$ is a Banach lattice. Then, the Banach space $X:=Y\oplus_1\mathbb{R}$ equipped with the coordinate-wise order is a Banach lattice and the functional $x^*:X\to \mathbb{R}$ given by $x^*(y,t):=t$ defines a non-trivial lattice homomorphism on it.

\noindent ($2$) $\Rightarrow$ ($3$) Let $T:E\to X$ be a lattice isomorphism onto a Banach lattice $X$ such that $\text{Hom}(X,\mathbb{R})\neq \{0\}$. By Proposition \ref{prop:characterization-isomorphic-Banach-lattices}, $\fbl[E]=\delta_E(E)\oplus \text{ker}(\widehat{T})$. Given a non-zero lattice homomorphism $x^*$ on $X$, we deduce from Proposition \ref{prop:isomorphic-enumeration-BLs} that $\text{ker}(\widehat{T})\subset \text{ker}(\widehat{T^*x^*})$.

\noindent ($3$) $\Rightarrow$ ($4$) Suppose that there exists $x_1^*\in S_{E^*}$ and an ideal $I$ in $\fbl[E]=\delta_E(E)\oplus I$ such that $I\subset \text{ker}(\widehat{x_1^*})$. Fix any $x_2^*\in S_{E^*}$, $x_2^*\neq x_1^*$. As $\text{ker}(x_1^*)$, $\text{ker}(x_2^*)$ are both closed hyperplanes of $E$, there is an isomorphism $T:\text{ker}(x_1^*)\to \text{ker}(x_2^*)$ \cite[Exercise 2.7]{FHHMPZ-book} (in fact, all closed hyperplanes of a Banach space are mutually isomorphic with a uniform constant $\leq 25$ \cite[Lemma 9.5.4]{AA-book}). Now, take $x_1,x_2\in E$ such that $x_1^*(x_1)=x_2^*(x_2)=1$, and consider the operator $S:E\to E$ defined by
$$
Sx:=T\bigl(x-x_1^*(x)x_1\bigr)+x_1^*(x)x_2, \qquad x\in E.
$$
It is clear that $S$ is bounded. Moreover, it is bijective with inverse given by
$$
S^{-1}y:=T^{-1}\bigl(y-x_2^*(y)x_2\bigr)+x_2^*(y)x_1, \qquad y\in E.
$$
Thus, $\overline{S}:\fbl[E]\to \fbl[E]$, the unique lattice homomorphism satisfying $\overline{S}\delta_E=\delta_ES$, is a (surjective) lattice isomorphism. From the latter we deduce that $\overline{S}(I)$ is also an ideal in $\fbl[E]$ and it is immediate to check that $\overline{S}(\delta_E(E))=\delta_E(E)$, so we obtain the decomposition $\fbl[E]=\delta_E(E)\oplus \overline{S}(I)$. 

On the other hand, observe that $S^*x_2^*=x_1^*$, given that for every $x\in E$ we have
$$
S^*x_2^*(x)=x_2^*(Sx)=x_2^*\left(T\bigl(x-x_1^*(x)x_1\right)\bigr)+x_1^*(x)=x_1^*(x),
$$
where for the last equality one must remember that $T:\text{ker}(x_1^*)\to\text{ker}(x_2^*)$. Now, note that $\widehat{x_2^*}\overline{S}=\widehat{S^*x_2^*}=\widehat{x_1^*}$ and this implies that $\overline{S}\bigl(\text{ker}(\widehat{x_1^*})\bigr)\subset \text{ker}(\widehat{x_2^*})$. Consequently, $\overline{S}(I)\subset \text{ker}(\widehat{x_2^*})$.

\noindent ($3$) $\Rightarrow$ ($1$)  (It is obvious that ($4$) $\Rightarrow$ ($3$)). Suppose that there are $x^*\in E^*$, $x^*\neq 0$, and an ideal $I$ in $\fbl[E]$ such that $\fbl[E]=\delta_E(E)\oplus I$ with $I\subset \text{ker}(\widehat{x^*})$. In the same way as in Proposition \ref{prop:kernel-zero} (or also in Proposition \ref{prop:isomorphic-zeros}), we can consider the lattice homomorphism $\overline{x^*}:\fbl[E]/I\to \mathbb{R}$ defined by $\overline{x^*}(\overline{f})=f(x^*)$. Note that if $x^*(x)\neq 0$ for some $x\in E$, then $\overline{x^*}(\overline{\delta_x})=\delta_x(x^*)=x^*(x)\neq 0$ so, in particular, $\overline{x^*}\neq 0$. Therefore, $\text{ker}(\overline{x^*})$ is a closed hyperplane of the Banach lattice $X=\fbl[E]/I$ which is also a Banach lattice. Since $X$ is isomorphic to $E$, this concludes the proof.
\end{proof}

\section{More open questions}\label{sec:more questions}

\subsection{Primariness of the class of Banach lattices} Recall that a Banach space $E$ is said to be \emph{primary} if whenever $E=F\oplus G$, then either $E\approx F$ or $E\approx G$. Let us recall some Banach spaces which do have this property:
\begin{itemize}
    \item In \cite{P60}, Pe\l{}czy\'{n}ski proved that $c_0$ and $\ell_p$, for $1\leq p<\infty$, are \textit{prime} (every --infinite-dimensional-- complemented subspace is isomorphic to the whole space); afterwards, Lindenstrauss showed that $\ell_\infty$ is also a prime Banach space \cite{Li67}.
    \item Every separable $C(K)$-space is primary: for $K$ being compact metric uncountable, this is due to Lindenstrauss and Pe\l{}czy\'{n}ski \cite[Corollary 1 to Theorem 2.1]{LP71}; for a countable $K$ this was solved in \cite{AB77} by Alspach and Benyamini. With regard to the uncountable case, a stronger result due to Rosenthal \cite{Rosenthal} should be mentioned: any complemented subspace of $C[0,1]$ with non-separable dual must be isomorphic to $C[0,1]$.
    \item The first proof of the primariness of $L_p$, for any $1\leq p <\infty$, can be found in the seminar notes by Maurey from 1974 \cite{M74-75}, although the author attributes this result to Enflo. In \cite{AEO77} and \cite{ES79} it is mentioned that Maurey's proof is based on \textit{unpublished techniques} due to Enflo presented at a conference in 1973. Alternative proofs are given in \cite[Theorem 1.3]{AEO77}, for the case $1<p<\infty$, and in \cite[Corollary 5.4]{ES79}, for the case $p=1$).
\end{itemize}
Partly inspired by these results, a related more general open question is the following:

\begin{question}[Primariness of the class of Banach lattices]\label{question:primariness-BLs}
Let $X$ be a Banach lattice and suppose that we have a decomposition $X=Y\oplus Z$ into two infinite-dimensional Banach spaces $Y$ and $Z$. Must then at least one of the factors be isomorphic to a Banach lattice?  
\end{question}

This question could also be formulated for the class of $C(K)$-spaces and the class of $L_1$-spaces, that is, if a $C(K)$-space (resp. an $L_1$-space) is isomorphic to $Y\oplus Z$, must $Y$ or $Z$ be isomorphic to a $C(K)$-space (resp. an $L_1$-space)? As mentioned at the beginning of this section, both separable $C(K)$-spaces and separable $L_1$-spaces are in fact \textit{primary}. Nevertheless, the current question seems to be open in the non-separable setting. %The decomposition $C(K_\mathcal{B})=\textsf{PS}_2\oplus C(K_\mathcal{A})$ given in \cite{PS2}, where $\textsf{PS}_2$ is not isomorphic to a $C(K)$-space, does not disprove Question \ref{question:primariness-BLs}.

\subsection{Projections in AM-spaces}

We have already mentioned in Question \ref{question:complemented-script-l-infinity}, whether every $\mathcal{L}_\infty$-space which is complemented in a Banach lattice must be complemented in some AM-space.
 
We now know that the space $\textsf{PS}_2$ is not isomorphic to a Banach lattice \cite{dHMST23} so, in particular, it is not isomorphic to an AM-space. On the other hand, in \cite{B77} Benyamini gave an example of an AM-space which is not complemented in any $C(K)$-space. However, the following question remains open:

\begin{question}
Suppose that $X$ is an AM-space isomorphic to a complemented subspace of a $C(K)$-space. Must $X$ be isomorphic to a $C(K)$-space?    
\end{question}

%\subsection{Banach lattices could be complemented in spaces with extra symmetry} 
%Recall that every Banach space with an unconditional basis is isomorphic to a complemented subspace of a space with a symmetric basis \cite{Li72}. Motivated by this it is natural to pose the following:

%\begin{question}
%Is every Banach lattice isomorphic to a complemented subspace of a rearrangement invariant space?
%\end{question}

%The above question is perhaps also of interest in the case of a Banach lattice which is a K\"othe function space on some measure space $(\Omega,\Sigma,\mu)$, and one wonders whether it is complemented in a rearrangement invariant space over the same $(\Omega,\Sigma,\mu)$. 

\subsection{Complemented subspaces of spaces with unconditional basis}
Among the oldest questions concerning the structure of complemented subspaces is whether every complemented subspace of a Banach space with an unconditional basis must have an unconditional basis. The most relevant positive result in this direction is given in \cite{KW}: every $1$-complemented subspace of a complex Banach space with a $1$-unconditional basis also has a $1$-unconditional basis. Other relevant partial results can be found in \cite{EW76, W78} where under extra assumptions if $X\oplus Y$ is a decomposition of a Banach space with unconditional basis, then one can partition the basis to build bases both for $X$ and $Y$. 

However, a formally weaker version of this problem (motivated by Corollaries 2.3 and 2.4 of \cite{dHMST23}) is also open:

\begin{question}
If a \textit{Banach lattice} $X$ is isomorphic to a complemented subspace of a space with an unconditional basis, must $X$ have an unconditional basis?
\end{question}

In connection with this, let us recall that every Banach space with an unconditional basis is isomorphic to a complemented subspace of a space with a \emph{symmetric basis} \cite{Li72}. Motivated by this, it might be natural to wonder whether every separable Banach lattice must be isomorphic to a complemented subspace of a \emph{rearrangement invariant space}. This is not always the case as the following argument suggested by W. B. Johnson shows: Take a separable Banach lattice without the approximation property, such as the one constructed by Szankowski in \cite{Szankowski} (in fact this can be even taken as an appropriate sublattice of $\ell_r(L_p(0,1))$, with $1\leq r<p<\infty$). If this space were complemented in a rearrangement invariant space $X$, it would have the bounded approximation property (BAP), leading to a contradiction (observe that rearrangement invariant spaces do in fact have the MAP as a consequence of the density of simple functions). However, we have the following:

\begin{question}
Is every Banach lattice with BAP isomorphic to a complemented subspace of a rearrangement invariant space?
\end{question}

It is also worth mentioning here that it is not known whether every separable super-reflexive Banach lattice embeds into a separable super-reflexive rearrangement invariant space. 

%The above question is perhaps also of interest in the case of a Banach lattice which is a K\"othe function space on some measure space $(\Omega,\Sigma,\mu)$, and one wonders whether it is complemented in a rearrangement invariant space over the same $(\Omega,\Sigma,\mu)$. 

\subsection{Existence of complemented disjoint sequences}
Motivated by the study of disjointly homogeneous Banach lattices (where every pair of disjoint sequences share an equivalent subsequence), the following question was posed in \cite{FHSTT}:

\begin{question}
    Does every \textit{separable} Banach lattice contain a disjoint sequence whose span is complemented?
\end{question}

Note that if the Banach lattice were not assumed to be separable in the previous question, the answer would clearly be no: in $\ell_\infty$, disjoint sequences span (lattice-isometric) copies of $c_0$, so they cannot be complemented subspaces. This question has positive answer for non-reflexive spaces (as these contain either a complemented sublattice isomorphic to $c_0$ or $\ell_1$ \cite[Proposition 2.3.11 and Theorem 2.4.15]{M-N-book}) and for rearrangement invariant spaces (because the subspace generated by characteristic functions of pairwise disjoint sets is always complemented by means of a conditional expectation operator \cite[Theorem 2.a.4]{LT2-book}).

A closely related question, also motivated by the theory of indecomposable Banach spaces, would be the following: 

\begin{question}\label{question:sublattice-complemented}
    Does every infinite-dimensional Banach lattice contain an infinite-dimensional complemented sublattice?
\end{question}

Note that in the previous question we did not require such a sublattice to also have infinite codimension. In that case, the answer would be negative, as there exist indecomposable $C(K)$-spaces (various examples of which are mentioned in \cite{Koszmider-survey}). However, what is asked in Question \ref{question:sublattice-complemented} is trivially true for $C(K)$-spaces: $\text{ker}(\delta_t)$, where $\delta_t:C(K)\to\mathbb{R}$ is the evaluation at $t\in K$, is an infinite-dimensional complemented ideal (and thus a sublattice). A useful description of finite-codimensional sublattices can be found in \cite[Section 5]{BT24}. 

On the other hand, observe that Question \ref{question:sublattice-complemented} has an affirmative answer in the separable case. Indeed, if $X$ contains a lattice isomorphic copy of $c_0$, then this sublattice must be complemented; if $X$ does not contain $c_0$, then $X$ is order continuous so it has \textit{many projection bands} \cite[Theorem 1.a.13]{LT2-book}.

%To be isomorphic to an $L_1$-space is a three space property \cite[Theorem 3.4.b]{CG-book}. However, the three space problem for $C(K)$-spaces has a negative answer and this was observed by F. Cabello (see \cite[Theorem 3.5.b]{CG-book}).  Being isomorphic to a BL is not a 3SP, see, for instance \cite{KP79} (in \cite{JLS80} it is shown that this space fails GL-lust). \textcolor{blue}{Cuidado aquí. Hay distintas nociones de ser propiedad de tres espacios, y creo que en este libro en general la idea es ver si se preservan propiedades POR SUMAS TORCIDAS. Lo de Talagrand parece diferente} 

\section*{Acknowledgements}
We would like to thank W. B. Johnson for important clarifications and helpful remarks on the topic of this paper. We are also grateful to the anonymous referee for their insightful and valuable comments.

Research partially supported by grants PID2020-116398GB-I00 and CEX2023-001347-S funded by  MCIN/AEI/10.13039/501100011033. D. de Hevia also benefited from an FPU Grant FPU20/03334 from Ministerio de Ciencia, Innovación y Universidades. 

\section*{Declarations}

\subsection*{Funding} This work was supported by grants PID2020-116398GB-I00 and CEX2023-001347-S funded by  MCIN/AEI/10.13039/501100011033 (Agencia Estatal de Investigación). The first author has also received financial support from Ministerio de Ciencia, Innovación y Universidades through an FPU Grant.

\subsection*{Ethics approval} 

Not applicable.

\subsection*{Competing interests}

The authors have no competing interests to declare that are relevant to the content of this article.

\end{document}